\documentclass[review,onefignum,onetabnum]{siamonline1116}
\usepackage[all]{xy}
\usepackage{amsmath}
\usepackage{mathrsfs,amssymb,amsmath,booktabs,array,xcolor,url,color,soul,multirow,cite}
\usepackage{mathbbold,bbm}
\usepackage{arydshln}
\usepackage{graphicx}
\usepackage{enumitem}
\newcommand*{\Luref}[1]{\nameref{#1}}
\usepackage{lipsum}
\usepackage{amsfonts}
\usepackage{graphicx}
\usepackage{epstopdf}
\usepackage{algorithmic}
\ifpdf
  \DeclareGraphicsExtensions{.eps,.pdf,.png,.jpg}
\else
  \DeclareGraphicsExtensions{.eps}
\fi

\newsiamremark{remark}{Remark}
\newsiamremark{hypothesis}{Hypothesis}
\crefname{hypothesis}{Hypothesis}{Hypotheses}
\newsiamthm{claim}{Claim}

\headers{CRN Decomposition for stability}{Y. Lu, C. Gao and D. Dochain}

\title{Chemical Reaction Network Decomposition Technique for Stability Analysis\thanks{Submitted to the editors. \date{\today}.
\funding{This work was funded by the National Nature Science Foundation of China under Grant
No. 12071428, 11671418, and the Zhejiang Provincial Natural Science Foundation of China under Grant No. LZ20A010002.
}}}

\author{Yafei Lu\thanks{School of Mathematical Sciences, Zhejiang University (\email{11535029@zju.edu.cn (Y. Lu), gaochou@zju.edu.cn (C. Gao, Correspondence)).}}
\and Chuanhou Gao\footnotemark[2]
\and Denis Dochain\thanks{ICTEAM, UCLouvain, B\^{a}timent Euler, avenue Georges Lema\^{i}tre 4-6, 1348 Louvain-la-Neuve, Belgium (\email{denis.dochain@uclouvain.be (D. Dochain).}}}

\usepackage{amsopn}

\ifpdf
\hypersetup{
  pdftitle={Chemical Reaction Network Decomposition Technique for Stability Analysis},
  pdfauthor={Yafei~Lu, Chuanhou~Gao, Denis Dochain}
}
\fi

\def\dd{\text{d}}
\newtheorem{property}{Property}
\newtheorem{example}{Example}
\setlist[enumerate]{listparindent=\parindent}
\begin{document}

\maketitle

\begin{abstract}                          % Abstract of not more than
This paper develops the concept of decomposition for chemical reaction networks, based on which a network decomposition technique is proposed to capture the stability of large-scale networks characterized by a high number of species, high dimension, high deficiency, and/or non-weakly reversible structure. We present some sufficient conditions to capture the stability of a network (may possess any dimension, any deficiency, and/or any topological structure) when it can be decomposed into a complex balanced subnetwork and a few $1$-dimensional subnetworks (and/or a few two-species subnetworks), especially in the case when there are shared species in different subnetworks. The results cover encouraging applications on autocatalytic networks with some frequently-encountered biochemical reactions examples of interest, such as the autophosphorylation of PAK$1$ and Aurora B kinase, autocatalytic cycles originating from metabolism, etc.
\end{abstract}

% REQUIRED

% REQUIRED

\section{Introduction}

Chemical reaction networks (CRNs) have a pivotal role in various fields including chemistry, system biology, molecular biology as well as other potential applications. The research on CRNs aims at studying the connection between the networks' structure and underlying dynamical characteristics, e.g. oscillations, persistence, stability, etc. Thereinto, the issue of stability for CRNs, especially in the mass-action kinetics setting, has received considerable critical attention, which is also the concern in the current paper.

The study refers to the stability for CRNs with general structure is difficult in general while
a number of researchers \cite{Angeli2009A,Craciun2006Multiple,Feinberg1972Complex,Feinberg1995The,Finberg1987Chemical,Finberg1988Chemical,Szederke2011Finding} have reported that certain networks equipped with special structure (e.g. reversible, or weakly reversible) exhibit convergence behaviors. Horn and Jackson \cite{HornJackson1972General} proposed the well-known deficiency zero theorem, which stated that a weakly reversible deficiency zero mass action system (MAS) is complex balanced and has a unique equilibrium in every positive stoichiometric compatibility class. Meanwhile, each such equilibrium is proved to be asymptotically stable by taking the pseudo-Helmholtz function as a Lyapunov function. Later, this result was extended by Feinberg \cite{Finberg1988Chemical}, named as the deficiency one theorem, in which the deficiency can be relaxed, i.e., not necessarily equal to zero. Another way to reach the stability of equilibria for complex balanced MASs (also contains detailed balanced MASs) was introduced in \cite{van2013mathematical}, which established a compact mathematical formulation to exhibit the dynamics of the network through the complex graph.
%It has been shown that the equilibria of the complex balanced CRN (also contains detailed CRN) are asymptotically stable by taking the pseudo-Helmholtz function as a Lyapunov function.
%complex balanced CRN (also contains detailed CRN) has a unique positive equilibrium in every positive stoichiometric compatibility class, and meanwhile, each such equilibirum is proved to be asymptotic stability by taking the pseudo-Helmholtz function as a Lyapunov funtion.
Further, the studies of \cite{Johnston2011Linear,Johnston2012Dynamical} have applied the results on complex balanced MASs to more CRNs. It says that if an MAS can be transformed into a complex balanced MAS through the linear conjugacy approach, then the former shares the same stability as the latter.
Several other attempts have been made to the stability for more CRNs by constructing their corresponding Lyapunov functions. For instance, the scaling limits of
non-equilibrium potential \cite{Anderson2015Lyapunov} were taken as the Lyapunov function for birth-death processes from a microscopic perspective; the piecewise linear in rates Lyapunov function \cite{Alradhawi2016New} was put forward for some balanced MASs;
the paper \cite{Ke2019Complex} developed a generalized pseudo-Helmholtz function to the stability analysis for a general balanced MAS by taking advantage of the reconstruction and reverse reconstruction strategies. In the meanwhile, the generalized pseudo-Helmholtz function also succeeded in settling the stability for complex-balanced-produced MASs \cite{Wu2020A}. Recently, Fang and Gao \cite{Fang2015Lyapunov} proposed a systematic approach called Lyapunov function PDEs to the stability for general CRNs, which has been validated well on complex balanced, $1$-dimensional, special high dimensional and complex balanced produced CRNs so far.

However, except for the above special cases, in most cases, such as the biological systems, many CRNs are non-weakly reversible, of any deficiency, and of any arbitrary dimension. More importantly, when the dimensions of the networks become large, the associated stability problem becomes challenging. Especially, as the rapid development of systems biology and synthetic biology, studying the dynamical behaviors of complex and large biochemical reactions has drawn great attention.  Some investigators focuses on the properties of large-scale CRNs by joining small networks (e.g., cross-talk) or decomposing networks (e.g., inhibition), like identifiability \cite{gross2020joining}, multistationarity, stability \cite{Lu2020} and stationary distribution \cite{hoessly2019stationary}. They mainly use the information of the small parts to infer the properties of the large network synthesized by these parts.
In particular, these contributions will promote the research on cross-talk \cite{gruning2010regulatory,lopezotin2010the,decraene2009crosstalk}, which refers to the situation where two or more signaling pathways with the same components affect each other, and shows great application potential in medicine.
% when two or more CRNs with shared species or when large CRNs can be divided into several coupled small CRNs from the mathematical perspective.

In line with the above mentioned studies, this paper is devoted to deriving the stability of large-scale CRNs by decomposing them into low dimensional networks and/or parts having special structures (e.g. reversible, weakly reversible). Note that the networks we are concerned with are all non-weakly reversible, of any dimension, and of any deficiency. To be specific, we first show the results that if the decomposition of a CRN contains a complex balanced CRN and several $1$-dimensional/two-species CRNs, then we can determine its stability by utilizing the information given by subnetworks. i.e. the corresponding Lyapunov functions for the subnetworks. These results cover two basic cases (i) the subnetworks have disjoint species sets, and (ii) the subnetworks have shared species.
Further, such a dimension reduction decomposition approach is applied to analyze the stability for some practical biological reactions, including the autophosphorylation of PAK1 and Aurora B kinase, autocatalytic cycles, etc. Moreover, we deal with a common class of autocatalytic reactions as an application. It proves that the stability for such CRNs can be captured in terms of the characteristics of the small parts: two-species autocatalytic CRNs generated from decomposition. In some sense, the present study may offer some significant insights into the field of stability analysis to synthetic biology and artificial life.

%here we start from a common class of autocatalytic reactions, which is an active area of research on biological systems (e.g., DNA replications \cite{Plasson2011Autocatalyses}, population evolution, etc), life and artificial life \cite{Kauffman1995At,Hordijk2004,Virgo2016Complex} (to describe self-organizing systems), and chemistry.
%We prove that the stability for such CRNs can be captured in terms of the characteristics of the small parts: two-species autocatalytic CRNs coming from decomposition. Further, we continue to show the results that if the decomposition of a CRN contains a complex balanced CRN and several autocatalytic CRNs/ $1$-dimensional CRNs, then we can determine its stability utilizing the information given by subnetworks. In essence, the decomposition method we proposed is dimensionality reduction. In turn, we also give the stability results on the case that two or more specific CRNs mentioned above compose into a single large one. Remarkably, we apple our results to some biochemical systems, including the autophosphorylation of
%PAK1 and Aurora B kinase, autocatalytic cycles, etc. Therefore, in some sense, the present study may offer some significant insights into the field of stability analysis to synthetic biology and artificial life.

The remaining part of the paper proceeds as follows. %Section $2$ revisits some crucial concepts and known results about CRNs as well as stating the motivation of this study. 
Section $2$ states the motivation of this study.
In section $3$, we discuss the definition of decomposition and the decomposition technique that serves for capturing the stability of a CRN. Section $4$ suggests a kind of special decomposition that includes two-species subnetworks, and the sufficient condition to render stability is also presented in this special case. In section $5$, we consider some applications of the proposed methods to autocatalytic CRNs. Finally, section $6$ concludes the paper.

\noindent{\textbf{Notation:}}\\
\rule[1ex]{\columnwidth}{0.8pt}
\begin{description}
   \item[\hspace{0em}{$\mathbb{R}^n, \mathbb{R}^n_{\geq 0}, \mathbb{R}^n_{>0}$}]:
	$n$-dimensional real space, non-negative real space, positive real space, respectively.
	\item[\hspace{0em}{$\mathbb{Z}^n_{\geq 0}$}]: $n$-dimensional non-negative integer space.
	\item[\hspace{0em}{$x^{v_{\cdot i}}$}]: $x^{v_{\cdot i}}=\prod_{j=1}^{d}x_{j}^{v_{ji}}$, where $x\in \mathbb{R}^{d}, v_{\cdot i}\in\mathbb{Z}^{d}$ and $0^{0}=1$.
	%\item[\hspace{-0.5em}{$\frac{x}{y}$}]: %$\frac{x}{y}=(\frac{x_1}{y_1}, \cdots, \frac{x_n}{y_n}) $, where $x\in\mathbb{R}^{n}$, $y\in\mathbb{R}^{n}_{>0}$.
	\item[\hspace{0em}{$\mathrm{Ln}(x)$}]: $\mathrm{Ln}(x)=\left(\ln{x_{1}}, \cdots, \ln{{x}_{n}} \right)^{\top}$, where $x\in\mathbb{R}^{n}_{>{0}}$.
	\item[\hspace{0em}{$\mathscr{C}^i(\cdot~ ; *)$}]:
	 the set of $i$th continuous differentiable functions from "$\cdot$" to "*".
	 \item[\hspace{0em}{s.t.}]: such that.
	 \item[\hspace{0em}{$e_i$}]: unit vector with $i$th element being one while others being zero.
\end{description}
\rule[1ex]{\columnwidth}{0.8pt}

\section{Motivation statement}
A CRN $\mathcal{N}=(\mathcal{S,C,R})$ equipped with mass-action kinetics is called an MAS, denoted by $\mathcal{M}=(\mathcal{S,C,R,K})$, where $\mathcal{S}$ is the species set, $\mathcal{C}$ the complex set, $\mathcal{R}$ the reaction set, and $\mathcal{K}$ the reaction rate constant set. More terminologies about CRNs and MASs can be found in \Luref{ap1}. CRNs or MASs can be grouped according to the dimension of stoichoimetric subspace into 1-dimensional, 2-dimensional networks, etc., or according to structure into weakly reversible and non-weakly reversible networks. For an $\mathcal{M}=(\mathcal{S,C,R,K})$ with the dynamics of
\begin{equation}\label{eq:mas0}
\frac{\mathrm{d}x}{\mathrm{d}t}=\Gamma \Xi(x)=k_ix^{v_{.i}},~~~x\in \mathbb{R}_{\geq0}^{n},
\end{equation}
the stability issue in the Lyapunov sense is about how to construct an available Lyapunov function suggesting the equilibrium $x^*$ constrained by $\Gamma \Xi(x^*)=0$ stable. It was reported that there have been available Lyapunov functions for complex balanced networks \cite{HornJackson1972General} (i.e., the classical pseudo-Helmholtz free energy function \cref{eq:Helmholtz}) and 1-dimensional networks \cite{Fang2015Lyapunov} (i.e., \cref{sub1Lya}).

Although it becomes relatively routine to generate Lyapunov functions for the above two classes of networks, there are no systematic methods for others to construct the Lyapunov functions, especially for those extremely complicated ones (or large-scale) characterized by high number of species, high dimension, high deficiency, and/or non-special structure, like metabolic regulatory networks, gene regulatory networks in biological systems. It will pose a great challenge to perform the stability analysis for them. However if we observe those large-scale networks from the viewpoint of part or locality, the subnetwork may be very special, even to be weakly reversible and/or $1$-dimensional. We use the following example \cite{hoessly2019stationary} to illustrate our idea.

\begin{example}\label{examp2}
 The network route takes \begin{align}\label{eg:motivation}
&
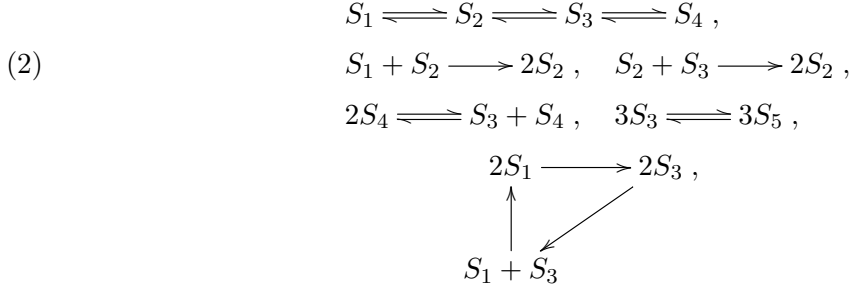
\xymatrix{S_{1} \ar @{ -^{>}}^{}  @< 1pt> [r]
	& S_{2} \ar  @{ -^{>}}^{}  @< 1pt> [l]
	\ar @{ -^{>}}^{}  @< 1pt> [r]
	&S_3 \ar  @{ -^{>}}^{}  @< 1pt> [l]
	\ar @{ -^{>}}^{}  @< 1pt> [r]
	&S_4 \ar  @{ -^{>}}^{}  @< 1pt> [l]},
\notag
\\
&\xymatrix{S_1 +S_2\ar[r]^-{}& 2S_2},~~
\xymatrix{S_2 +S_3\ar[r]^-{}& 2S_2},\\ \notag
&\xymatrix{
	2S_4\ar @{ -^{>}}^{}  @< 1pt> [r]& S_3+S_4 \ar  @{ -^{>}}^{}  @< 1pt> [l] },~~
\xymatrix{3S_3\ar @{ -^{>}}^{}  @< 1pt> [r]& 3S_5 \ar  @{ -^{>}}^{}  @< 1pt> [l]},\\\notag
&\qquad \qquad \xymatrix{2S_1\ar[r]^-{}&2S_3\ar[ld]^-{}\\
	\ar[u]^-{} S_1+S_3 &},	
\end{align}
which looks rather complicated, and has dimension $5$, deficiency $3$ and non-weakly reversible structure. However, its subnetworks, as an example of a kind of decomposition, may include a weakly reversible network on $\{S_1,S_3,S_5\}$ given by
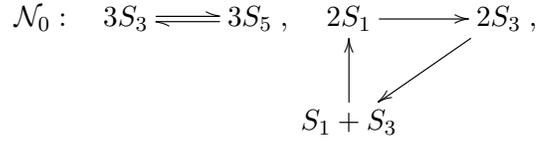
\begin{align*}
	\mathcal{N}_0:~~\xymatrix{3S_3\ar @{ -^{>}}^{}  @< 1pt> [r]& 3S_5 \ar  @{ -^{>}}^{}  @< 1pt> [l]},
	\xymatrix{2S_1\ar[r]^-{}&2S_3\ar[ld]^-{}\\
		\ar[u]^-{} S_1+S_3 &},
\end{align*}
and three $1$-dimensional CRNs on $\{S_1,S_2\}, \{S_2,S_3\}, \{S_3,S_4\}$ successively,
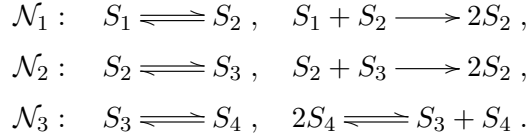
\begin{align*}
	&\mathcal{N}_1:~~\xymatrix{S_{1} \ar @{ -^{>}}^{}  @< 1pt> [r]
		& S_{2}\ar  @{ -^{>}}^{}  @< 1pt> [l]},~~
	\xymatrix{S_1 +S_2\ar[r]^-{}& 2S_2},\\
	&\mathcal{N}_{2}:~~\xymatrix{ S_{2}
		\ar @{ -^{>}}^{}  @< 1pt> [r]
		&S_3 \ar  @{ -^{>}}^{}  @< 1pt> [l]},~~
	\xymatrix{S_2 +S_3\ar[r]^-{}& 2S_2},\\
	&\mathcal{N}_3:~~\xymatrix{S_3\ar @{ -^{>}}^{}  @< 1pt> [r]& S_4 \ar  @{ -^{>}}^{}  @< 1pt> [l]},~~
	\xymatrix{2S_4\ar @{ -^{>}}^-{}  @< 1pt> [r]& S_3+S_4 \ar  @{ -^{>}}^-{}  @< 1pt> [l]}.
\end{align*}
\end{example}
If every subnetwork is stable (relatively easy to be checked), is it possible for the original network to be stable too, or can the Lyapunov function for the original network be derived from those known ones for special subnetworks? This naive idea motivates us to analyze the stability of a large-scale or complicated network through decomposing it into some small-scale and/or simple subnetworks for which the Lyapunov functions are known.

\section{Decomposition technique}\label{sec3}
There are two possibilities in decomposing a network into some subnetworks: one is that all subnetworks have no disjoint species, the other is that there are shared species among subnetworks. We discuss these two cases in this section, but give the definition of decomposition first.

\subsection{Decomposition}
\begin{definition}\emph{(Decomposition).} \label{decomposition}
For an $\mathcal{M}=(\mathcal{S,C,R,K})$ given by \cref{eq:mas0} that admits an equilibrium $x^*\in\mathbb{R}^{n}_{>0}$, if there are finitely many $\mathcal{M}^{(p)}=(\mathcal{S}^{(p)},\mathcal{C}^{(p)},\mathcal{R}^{(p)},\mathcal{K}^{(p)})$ satisfying
\begin{enumerate}[itemindent=3em]
	\item [\emph{(1)}]  $\mathcal{S}=\cup~\mathcal{S}^{(p)}, \mathcal{C}=\cup~\mathcal{C}^{(p)}, \mathcal{R}=\cup~\mathcal{R}^{(p)},\mathcal{K}=\cup~\mathcal{K}^{(p)}$;
	\item [\emph{(2)}] every $\mathcal{M}^{(p)}=(\mathcal{S}^{(p)},\mathcal{C}^{(p)},\mathcal{R}^{(p)},\mathcal{K}^{(p)})$ admits an equilibrium $x^{*^{(p)}}\in \mathbb{R}^{n_p}_{>0}$, in which each entry is the same as that in $x^*$ if they represent the concentration of the same species in $\mathcal{S}$,
\end{enumerate}
then $\mathcal{N}=(\mathcal{S,C,R})$ is said to be decomposable.
\end{definition}

Based on the decomposition definition, we have the following property about the equilibrium.

\begin{property}\label{pro:eq1}
Consider an $\mathcal{M}=(\mathcal{S,C,R,K})$ governed by \cref{eq:mas0}, and $x^*\in\mathbb{R}^{n}_{>0}$ is an equilibrium in $\mathcal{M}$. Assume $\mathcal{M}$ can be decomposed into a complex balanced $\mathcal{M}^{(0)}=(\mathcal{S}^{(0)},\mathcal{C}^{(0)},\mathcal{R}^{(0)},\mathcal{K}^{(0)})$ and $\ell$ 1-dimensional $\mathcal{M}^{(p)}=(\mathcal{S}^{(p)},\mathcal{C}^{(p)},\mathcal{R}^{(p)},\mathcal{K}^{(p)})$, $p=1,...,\ell$. If $\forall p\in\{1,\cdots,\ell\}$, $x^{*^{(p)}}\in\mathbb{R}^{n_p}_{>0}$, as defined in \cref{decomposition}, is a reaction vector balanced equilibrium in $\mathcal{M}^{(p)}$, then $x^*$ is a generalized balanced equilibrium in $\mathcal{M}$.
\end{property}

\begin{proof}
Let each entry of $x^{*^{(0)}}\in\mathbb{R}^{n_0}_{>0}$ equals to the entry of $x^*$ pointing to the same species in $\mathcal{S}$, then $x^{*^{(0)}}$ is a complex balanced equilibrium in $\mathcal{M}^{(0)}$. Based on \cref{def:generalized balanced} and \cref{decomposition}, it is straightforward to reach the conclusion.
\end{proof}

Note that the network decomposition will not generate or lose species, complex and reaction. Since every subnetwork is relatively simple and smaller-scale compared with the original network, it is possible to analyze the stability of every subnetwork and accordingly infer that of the original network. There are two cases when implementing network composition, one is $\forall p\neq q,~ \mathcal{S}^{(p)} \cap \mathcal{S}^{(q)}=\emptyset$, the other is $\mathcal{S}^{(p)} \cap \mathcal{S}^{(q)}\neq\emptyset$.
In the following, we will discuss these two situations respectively, and moreover, the subnetworks are limited to be complex balanced and $1$-dimensional for convenience of stability analysis.

\subsection{Decomposition: subnetworks with disjoint species sets}
This case is relatively trivial, especially when the subnetworks are complex balanced and $1$-dimensional. In the following we directly give a conclusion.

\begin{theorem}\label{thm:disjoint}
	For an $\mathcal{M}=(\mathcal{S,C,R,K})$ governed by \cref{eq:mas0}, let $x^*\in\mathbb{R}^{n}_{>0}$ be an equilibrium in $\mathcal{M}$. Suppose that $\mathcal{M}$ can be decomposed into a complex balanced $\mathcal{M}^{(0)}=(\mathcal{S}^{(0)},\mathcal{C}^{(0)},\mathcal{R}^{(0)},\mathcal{K}^{(0)})$ and $\ell$ 1-dimensional $\mathcal{M}^{(p)}=(\mathcal{S}^{(p)},\mathcal{C}^{(p)},\mathcal{R}^{(p)},\mathcal{K}^{(p)})$, $p=1,...,\ell$, and moreover, $\forall p\neq q\in\{0,1,...,\ell\},~ \mathcal{S}^{(p)} \cap \mathcal{S}^{(q)}=\emptyset$. If for every 1-dimensional subnetwork there is $\omega^\top_p \frac{\partial}{\partial{x^{(p)}}} h(x^{*^{(p)}},1)<0$, then $\mathcal{M}$ is locally asymptotically stable at $x^*$, where $h(x^{*^{(p)}},1)$ and $\omega_p$ follow the definition of \cref{eq:sub1Lya}, and $x^{*^{(p)}}\in \mathbb{R}^{n_p}_{>0}$ is an equilibrium in $\mathcal{M}^{(p)}$, $p=1,...,\ell,~ n_0+\sum_{p=1}^{\ell}n_p=n.$
\end{theorem}

\begin{proof}
The result is similar to Theorem 27 in \cite{Fang2015Lyapunov} where a composition network containing a complex balanced MAS and some 1-dimensional MASs is proved asymptotically stable, and we thus omit it here.
\end{proof}

\subsection{Decomposition: subnetworks with shared species}\label{sec:3.3}
The situation will become intractable if the subnetworks share species in network decomposition. In this case, the dynamic behaviors of the subsystems will affect each other through the shared species. However, we manage to derive a sufficient condition to suggest the asymptotic stability of a class of networks that can be decomposed into a complex balanced subnetwork and a few $1$-dimensional subnetworks but with species shared among subnetworks.

\begin{theorem}\label{thm:com-1}
	Given an $\mathcal{M}=(\mathcal{S,C,R,K})$ with the dynamics \cref{eq:mas0} and an equilibrium $x^*\in\mathbb{R}^n_{>0}$, assume it can be decomposed into a complex balanced $\mathcal{M}^{(0)}$ and $\ell$ 1-dimensional $\mathcal{M}^{(p)}$'s, $p=1,\cdots,\ell$, and moreover,  $\forall p,q\in\{1,\cdots,\ell\}$, $E_p:= \mathcal{S}^{(0)}\bigcap\mathcal{S}^{(p)}\neq\emptyset$ , $\mathcal{S}^{(p)}\bigcap \mathcal{S}^{(q)}\in \bigcup^{\ell}_{p=1}E_p$ or $\emptyset$, and every $\mathcal{M}^{(p)}$ is reaction vector balanced. Then $\mathcal{M}$ is locally asymptotically stable at $x^*$ if for every $p\in\{1,...,\ell\}$ the following conditions are true
	\begin{itemize}[itemindent=3em]
		\item [\emph{(1)}] for all $S_j\in E_p$ in $v^{(p)}_{\cdot l}\longrightarrow v'^{(p)}_{\cdot l}$ there is $v^{(p)}_{\cdot m}\longrightarrow v'^{(p)}_{\cdot m}$ such that $v^{(p)}_{jm}=v'^{(p)}_{jl},v'^{(p)}_{jm}=v^{(p)}_{jl}$, and $v'^{(p)}_{jm}-v^{(p)}_{jm}=v^{(p)}_{jl}-v'^{(p)}_{jl}=1$;
		\\
			\item[\emph{(2)}] denote 	$\tilde{x}^{(p)}=(\bigotimes_{S_j\in \mathcal{S}^{(p)}\backslash E_p} x^{(p)}_j)^\top$,
		$L_{\omega_p}=\{l:v'^{(p)}_{\cdot l}- v^{(p)}_{\cdot l}=(\bigotimes_{S_i\in E_p}{1}, \omega^\top_p)^\top\}$, and $R_{\omega_p}=\{l:v'^{(p)}_{\cdot l}- v^{(p)}_{\cdot l}=-(\bigotimes_{S_i\in E_p}{1}, \omega^\top_p)^\top\}$, it holds that
		\begin{align}\label{thm:con-com-1}
			\omega^\top_p \nabla \tilde{u}_p(\tilde{x}^{(p)})|_{\tilde{x}^{(p)}=\tilde{x}^{*^{(p)}}}>0,
		\end{align}
		where
		\begin{align*}
			\tilde{u}_p(\tilde{x}^{(p)})= \frac{\prod_{S_i\in E_p}{x^*_i}^{(p)}\sum_{R_{\omega_p}}k^{(p)}_{l}\prod_{S_j\in \mathcal{S}^{(p)}\backslash E_p} {x^{(p)}_j}^{v^{(p)}_{jl}}}{\sum_{L_{\omega_p}}k^{(p)}_{l}\prod_{S_j\in \mathcal{S}^{(p)}\backslash E_p} {x^{(p)}_j}^{v^{(p)}_{jl}}}.
		\end{align*}
	
	\end{itemize}
\end{theorem}

\begin{proof}
	The detailed proof can be found in Appendix 2. \Luref{Thm3.3proof}.
\end{proof}

\section{Special decomposition: a complex balanced subnetwork and some two-species subnetworks with shared species}\label{sec4}
In this section, we weaken the conditions presented in \cref{thm:com-1} by considering the decomposition including two-species subnetworks, which are $1$-dimensional essentially, instead of general 1-dimensional networks. Further, we extend the previous decomposition method to deal with more complex CRNs for capturing stability.

\subsection{Two-species CRNs and stability}\label{sec:4.1}
Consider a class of two-species CRNs with dimension $1$ defined on the species set $\{S_i, S_j\}$, given by
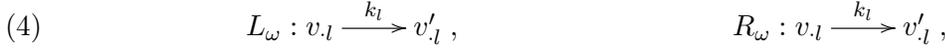
\begin{align}
&\xymatrix{L_{\omega}: v_{\cdot l} \ar[r]^-{k_{l}}  & v'_{\cdot l}},~~~
		&\xymatrix{R_{\omega}:v_{\cdot l} \ar[r]^-{k_{l}}  & v'_{\cdot l}},
\end{align}
 where $L_{\omega}=\{l:v'_{\cdot l}- v_{\cdot l}=\omega^\top=(w_i,w_j)\}$, and $R_{\omega}=\{l:v'_{\cdot l}- v_{\cdot l}=-\omega^\top=-(w_i,w_j)\}$. This class of CRNs satisfies that each $v_{il}=a$ $\forall l$ in $L_{\omega}$ and each $v_{jl}=b$ $\forall l$ in $R_{\omega}$.

 The stability of the two-species MASs can be naturally reached by adopting the Lyapunov function given in \cref{sub1Lya} for all $1$-dimensional MASs. However, we propose another one below for the same purpose, which yet will play an important role on weakening the conditions of \cref{thm:com-1} when implementing the network decomposition technique.

 \begin{lemma}\label{le:tw}
For a class of two-species CRNs defined above, suppose it is reaction vector balanced with an equilibrium $x^*=(x^*_i,x^*_j)\in \mathbb{R}^2_{>0}$, then the function
\begin{align}
f(x)=-\int^{x_i}_{x^*_i}w^{-1}_i\ln \frac{t^a}{c_{ij}\sum_{ R_{\omega}}k_lt^{v_{il}}}dt
		+
		\int^{x_j}_{x^*_j}w^{-1}_j\ln \frac{t^b}{c_{ij}\sum_{ L_{\omega}}k_lt^{v_{jl}}}dt
\end{align}
	with $c_{ij}=\frac{x^{*a}_i}{\sum_{ R_{\omega}}k_lx_i^{*^{v_{il}}}}=\frac{x^{*b}_i}{\sum_{ L_{\omega}}k_lx_j^{*^{v_{jl}}}}$
	can act as a Lyapunov function to suggest that
	$x^*$ is locally asymptotically stable if there are
	\begin{align}\label{con1}
		w^{-1}_i\sum_{R_{\omega}} k_l(a-v_{il})x^{*^{v_{il}-1}}_{i}<0
	\end{align}
	and
	\begin{align}\label{con2}
		w^{-1}_j\sum_{L_{\omega}} k_l(b-v_{jl})x^{*^{v_{jl}-1}}_{j}>0.
	\end{align}
	%\end{enumerate}
\end{lemma}
\begin{proof}
Firstly, we verify that $f(x)$ is positive definite.
The Hessian matrix of $f(x)$ is calculated by
\begin{align*}
	\nabla^2 f(x)
	=\text{diag}\left(-\frac{w^{-1}_i\sum_{R_{\omega}} k_l(a-v_{il})x^{v_{il}-1}_{i}}{\sum_{R_{\omega}} k_lx^{v_{il}}_{i}},
	\frac{w^{-1}_j\sum_{L_{\omega}} k_l(b-v_{jl})x^{v_{jl}-1}_{j}}{\sum_{L_{\omega}} k_lx^{v_{jl}}_{j}}\right).
\end{align*}
Based on the continuity of $\nabla^2 f(x)$ with respect to $x$ and conditions \cref{con1},\cref{con2}, there exists a neighborhood of $x^*$, denoted by $D(x^*)$, such that $\forall x\in D(x^*)\bigcap \mathscr{S}^+(x^*)$, there is $\nabla^2 f(x)>0$, which means $f(x)$ is locally strictly convex. This together with the fact $\nabla f(x^*)=\left(-w^{-1}_i\ln \frac{x^{*a}_i}{c_{ij}\sum_{ R_{\omega}}k_lx_i^{*^{v_{il}}}},w^{-1}_j\ln \frac{x^{*b}_j}{c_{ij}\sum_{ L_{\omega}}k_lx_j^{*^{v_{jl}}}}\right)^\top=(0,0)^\top$ yields $f(x^*)=0$ and $f(x)>0$ for any $x\in D(x^*)\bigcap \mathscr{S}^+(x^*)$ but $x\neq x^*$.

Next we prove that $f(x)$ fulfills dissipativeness. The dynamics of the two-species MAS are given by
\begin{align}\label{dynamics-tw}
\left\{
	\begin{array}{ll}
		\dot{x}_i=w_i\sum_{{L}_{\omega}}k_lx^a_i x^{v_{jl}}_j-w_i\sum_{ {R}_{\omega}}k_lx^b_j x^{v_{il}}_i,
		\\
		\dot{x}_j=w_j\sum_{ {L}_{\omega}}k_lx^a_i x^{v_{jl}}_j-w_j\sum_{ {R}_{\omega}}k_lx^b_j x^{v_{il}}_i.
	\end{array}
	\right.	
\end{align}
Then we have
\begin{align}\label{eq:dissipative}
\dot {f}(x)&=\nabla^\top f(x) \dot {x}
\notag
\\
&=-\ln \frac{\sum_{{L}_{\omega}}k_lx^a_i x^{v_{jl}}_j}{\sum_{ {R}_{\omega}}k_lx^b_j x^{v_{il}}_i}
\bigg(
\sum_{{L}_{\omega}}k_lx^a_i x^{v_{jl}}_j-\sum_{ {R}_{\omega}}k_lx^b_j x^{v_{il}}_i
\bigg)
\notag
\\
&\leq0.
\end{align}
Here, the last equality holds if and only if
$$\sum_{{L}_{\omega}}k_lx^a_i x^{v_{jl}}_j-\sum_{ {R}_{\omega}}k_lx^b_j x^{v_{il}}_i=0,$$ which means $x=x^*$.
Therefore, $f(x)$ can be used as a Lyapunov function to capture the local asymptotic stability of $x^*$.
\end{proof}

In the following, we consider a special class of two-species networks, i.e., two-species autocatalytic CRNs. Autocatalytic CRNs is an active area of research on biological systems (e.g., DNA replications \cite{Plasson2011Autocatalyses}, population evolution, etc), life and artificial life \cite{Kauffman1995At,Hordijk2004,Virgo2016Complex} (to describe self-organizing systems), and chemistry.
The notion of autocatalytic reactions is usually used to describe a class of reactions that are catalyzed by the products.
\begin{definition} [a reduced version of autocatalytic CRNs \cite{hoessly2019stationary}]\label{auto}
	An MAS $\mathcal{M}=(\mathcal{S,C,R,K})$ is said to be autocatalytic if it satisfies
	\begin{enumerate} [itemindent=3em]
		\item[\emph{(1)}] all reactions have a net consumption of one $S_i$ and a net production of one $S_j$, i.e., in the form of
		
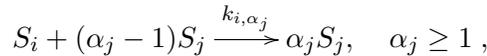
\begin{equation*}
			\xymatrix{S_i+(\alpha_j-1) S_j \ar[r]^-{k_{i,\alpha_j}}  & \alpha_j S_j, ~~~\alpha_j\geq 1},
			\end{equation*}
				and for convenience, denote the collection of the above reactions by $$\mathcal{R}_{i,j}:=\{v_{\cdot l}\longrightarrow v'_{\cdot l}\in \mathcal{R}:v'_{\cdot l}-v_{\cdot l}=e_j-e_i\},$$ where $e_i$ and $e_j$ are unit vectors;
			\item[\emph{(2)}] there must exist a pair of monomolecular reversible reactions;
			\item[\emph{(3)}] if $S_i\longrightarrow S_j \in \mathcal{R}_{i,j}\subset\mathcal{R}$ and $S_l\longrightarrow S_j \in \mathcal{R}_{l,j}\subset\mathcal{R}$, then the reactions in $\mathcal{R}_{i,j}$ and $\mathcal{R}_{l,j}$ contain reactions of the same molecularity such that there exists some $c\in \mathbb{R}_{>0}$ with $$c \cdot (k_{i,1},...,k_{i,n_j})=(k_{l,1},...,k_{l,n_j}),$$
		where $n_j$ represents the highest integer of the reaction $S_i+(n_j-1)S_j\longrightarrow n_j S_j$.
	\end{enumerate}
	\end{definition}
	
Any two-species autocatalytic CRN is a $1$-dimensional network, and is thus locally asymptotically stable at its equilibrium based on \cref{le:tw}.

\begin{corollary}\label{le:auto}
	For any two-species autocatalytic $\mathcal{M}=(\mathcal{S,C,R,K})$ governed by Eq. \cref{dynamics-tw} with $L_{\omega}=\mathcal{R}_{i,j}$, $R_{\omega}=\mathcal{R}_{j,i}$, $w_i=-1, w_j=1$, assume it admits a reaction vector balanced equilibrium $x^*\in \mathbb{R}^2_{>0}$,
	%then we have
	%\begin{enumerate}
	%\item[(1).]
	%$c_{ij}:=\frac{x^*_i}{\sum_{ \mathcal{R}_{j,i}}k_{j,\alpha_i}x^{*^{\alpha_i-1}}_i}=\frac{x^*_j}{\sum_{\mathcal{R}_{i,j}}k_{j,\alpha_j}x^{*^{\alpha_j-1}}_j}$,
	%\item[(2).]
	the function
	\begin{align}\label{eq:Lyauto}
		f(x)=\int^{x_i}_{x^*_i}\ln \frac{x_i}{c_{ij}\sum_{ \mathcal{R}_{j,i}}k_{j,\alpha_i}x^{\alpha_i-1}_i} dx_i
		+
		\int^{x_j}_{x^*_j}\ln \frac{x_j}{c_{ij}\sum_{ \mathcal{R}_{i,j}}k_{i,\alpha_j}x^{\alpha_j-1}_j} dx_j
	\end{align}
	with $c_{ij}=\frac{x^*_i}{\sum_{ \mathcal{R}_{j,i}}k_{j,\alpha_i}x^{*^{\alpha_i-1}}_i}=\frac{x^*_j}{\sum_{\mathcal{R}_{i,j}}k_{i,\alpha_j}x^{*^{\alpha_j-1}}_j}$
	can behave as a Lyapunov function to show 
	$x^*$ is locally asymptotically stable if
	\begin{align}\label{con1-auto}
		\sum_{ v_{\cdot l}\longrightarrow v'_{\cdot l} \in \mathcal{R}_{i,j}} k_{{i,\alpha_j}}(2-\alpha_j)x^{*^{\alpha_j-1}}_{j}>0
	\end{align}
	and
	\begin{align}\label{con2-auto}
		\sum_{ v_{\cdot l}\longrightarrow v'_{\cdot l} \in \mathcal{R}_{j,i}} k_{{j,\alpha_i}}(2-\alpha_i)x^{*^{\alpha_i-1}}_{i}>0
	\end{align}
	%\end{enumerate}
	are ture.
\end{corollary}

\begin{proof}
By applying \cref{le:tw} to the two-species autocatalytic MAS, the result is obvious.
\end{proof}

\begin{example}
	Given a two-species autocatalytic network with deficiency $4$, which has the reaction route
	\begin{align}
		&
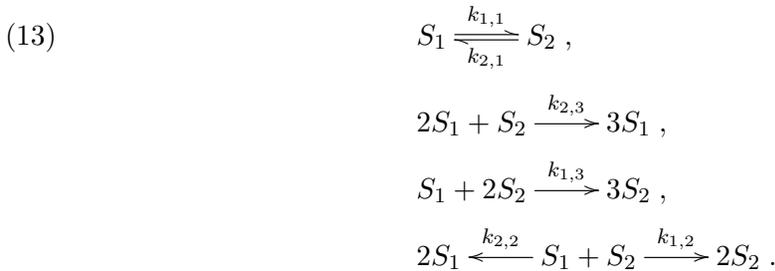
\xymatrix{S_1 \ar @{ -^{>}}^{k_{1,1}}  @< 1pt> [r]& S_2 \ar  @{ -^{>}}^{k_{2,1}}  @< 1pt> [l]},\\\notag
		&\xymatrix{2S_1+S_2\ar[r]^-{k_{2,3}}&3S_1},\\\notag
		&\xymatrix{S_1+2S_2 \ar[r]^-{ k_{1,3}} & 3S_2},\\\notag
		%&\xymatrix{&2S_2\\
		%	S_1+S_2 ~~\ar[ur]^-{k_{1,2}}~~\ar[dr]^-{k_{2,2}}&\\
		%&2S_1}
		&\xymatrix{2S_1&\ar[l]_-{k_{2,2}}S_1+S_2 \ar[r]^-{k_{1,2}}&2S_2}.
	\end{align}
	The type of autocatalytic CRN can be used to model the foraging system in ants for the research on decision-making problems \cite{khaluf2017the}. Here, we focus on the properties of its equilibria.
	By choosing $k_{1,3}=k_{2,3}= k_{1,2}=1, k_{2,2}=3, k_{2,1}=2, k_{1,1}=4$, the dynamical equations are as follows
	\begin{align*}
		\left\{
		\begin{array}{ll}
			\dot{x}_1=-(4x_1+x_1  x^2_2+x_1 x_2)+2x_2+x_2 x^2_1+3x_1 x_2,\\
			\dot{x}_2=4x_1+x_1  x^2_2+x_1 x_2-(2x_2+x_2 x^2_1+3x_1 x_2).
		\end{array}
		\right.
	\end{align*}
	We can calculate an equilibrium $(1,1)^\top$ within the positive compatibility class $\{x_1+x_2=2\}$. According to \cref{le:auto}, we can find a candidate Lyapunov function like Eq. \cref{eq:Lyauto} for the system
	\begin{align*}
		f(x)=\int^{x_1}_{1}\ln \frac{6t}{t^2+3t+2} dt+\int^{x_2}_{1}\ln \frac{6t}{t^2+t+4} dt.
	\end{align*}
	Then it is easy to validate that
	$(1,1)^\top$ satisfies Eqs. \cref{con1} and \cref{con2}, that is
	$$2-x^{*^2}_1|_{x^{*}_1=1}=1>0,$$ and $$4-x^{*^2}_2|_{x^{*}_2=1}=3>0.$$
	Therefore, $(1,1)^\top$ is locally asymptotically stable.
\end{example}

\begin{remark}\label{re:stability}
	If a two-species autocatalytic CRN is at-most-biomolecular (the sum of stoichiometric coefficients of any complex in the network is at most 2 \cite{gross2020joining}), i.e., every $\alpha_i \leq2$ and $\alpha_j\leq2$, then this system must be locally asymptotically stable since Eqs. \cref{con1} and \cref{con2} are satisfied naturally.
\end{remark}
\subsection{CRNs decomposed into a complex balanced subnetwork and some two-species subnetworks}
When the decomposition contains two-species networks instead of general $1$-dimensional networks, we have the following result.

\begin{theorem}\label{thm:com-tw}
For an $\mathcal{M}=(\mathcal{S,C,R,K})$ described by \cref{eq:mas0} with an equilibrium $x^*\in\mathbb{R}^n_{>0}$, assume it can be broken into a complex balanced $\mathcal{M}^{(0)}=(\mathcal{S}^{(0)},\mathcal{C}^{(0)},\mathcal{R}^{(0)},\mathcal{K}^{(0)})$ and $\ell$ reaction vector balanced two-species $\mathcal{M}^{(p)}=(\mathcal{S}^{(p)},\mathcal{C}^{(p)},\mathcal{R}^{(p)},\mathcal{K}^{(p)}),$ $p=1,...,\ell$, and moreover, $\forall p,~\mathcal{S}^{(p)}=\{S_i,S_j\}\subseteq\mathcal{S}$ and $ E_p:= \mathcal{S}^{(0)}\bigcap\mathcal{S}^{(p)}\neq\emptyset$. Then the local asymptotic stability of $x^*$ can be achieved if the following conditions are true,
	\begin{itemize}[itemindent=3em]
			\item[\emph{(1)}] let $E=\bigcup^{\ell}_{p=1}E_p,$ for every $S_i\in E$, the reaction in
			$R_{w_p}$ in $\mathcal{M}^{(p)}$ satisfies $v_{il}^{(p)}-a^{(p)}=w_{ip}$,
		
		\item[\emph{(2)}] for any species $S_j\in \mathcal{S}^{(p)}\bigcap \mathcal{S}^{(q)}$, $p,q \in \{1,\cdots,\ell\}$, the reactions in $\mathcal{M}^{(p)}$ and $\mathcal{M}^{(q)}$ contain the same
stoichiometric coefficient of $S_j$ such that there is a constant $c\in\mathbb{R}_{>0}$ with $(k^{(p)}_1,\cdots,k^{(p)}_{r_p})=c\cdot(k^{(q)}_1,\cdots,k^{(q)}_{r_q}), r_p=r_q$. 
		
		\item[\emph{(3)}] $\forall S_j\in \mathcal{S}\backslash \mathcal{S}^{(0)}$ and $\forall p$, there holds
		\begin{align}\label{eq:com-tw-con}
		w^{-1}_{jp}\sum_{L_{\omega_p}} k^{(p)}_l(b^{(p)}-v_{jl}^{(p)})x^{*^{v_{jl}^{(p)}-1}}_{j}>0.
		\end{align}
	\end{itemize}
\end{theorem}
\begin{proof}
The detailed proof can be found in Appendix 2. \Luref{Thm16proof}.
\end{proof}

\begin{remark}\label{re:com-1}
Unlike \cref{thm:com-1}, the above result suggests that the decomposed two-species CRNs may share more species but not just the species included in $\mathcal{S}^{(0)}$.
\end{remark}

By combing \cref{thm:com-tw} with \cref{thm:com-1}, we can easily obtain another decomposition pattern that may be applied to more CRNs.

\begin{corollary}\label{exthm:com-1}
	Given a $\mathcal{M}=(\mathcal{S,C,R,K})$ governed by \cref{eq:mas0} with an equilibrium $x^*\in\mathbb{R}^n_{>0}$, assume it can be decomposed into a complex balanced $\mathcal{M}^{(0)}$ and $\ell$ 1-dimensional $\mathcal{M}^{(p)}$'s satisfying reaction vector balancing. Then the system is locally asymptotically stable at $x^*$ if
		\begin{itemize}[itemindent=3em]
	    \item[\emph{(1)}] $\forall p,q\in\{1,\cdots,\ell\}$,  when $\mathcal{M}^{(p)}, \mathcal{M}^{(q)}$ are both two-species CRNs and all their reactions satisfy condition $(2)$ in \cref{thm:com-tw}; when $\mathcal{M}^{(p)}, \mathcal{M}^{(q)}$ are two-species CRNs but do not satisfy condition $(2)$ in \cref{thm:com-tw}, or they are not two species CRNs, there is $\mathcal{S}^{(p)}\bigcap \mathcal{S}^{(q)}\in \bigcup^{\ell}_{p=1}E_p$ or $\emptyset$ with $E_p:= \mathcal{S}^{(0)}\bigcap\mathcal{S}^{(p)}\neq\emptyset$;
		\item[\emph{(2)}] $\forall p=1,\cdots,\ell$, for those $\mathcal{M}^{(p)}$s that are two-species CRNs and satisfy condition $(2)$ in \cref{thm:com-tw}, conditions $(1)$ and $(3)$ in \cref{thm:com-tw} is satisfied,  and for others conditions $(1)$ and $(2)$ in \cref{thm:com-1} are satisfied.	
	\end{itemize}
\end{corollary}

We revisit the network \cref{eg:motivation} in \cref{examp2} to illustrate \cref{thm:com-tw} and \cref{exthm:com-1}.

\begin{example}\label{eg:7}
	%Consider the following $\mathcal{M}$ (example 5.5, in\cite{hoessly2019stationary}) of $5$-dimenaional and with deficiency $3$,
	%\begin{align}
	%&\xymatrix{S_{1} \ar @{ -^{>}}^{}  @< 1pt> [r]
	%& S_{2} \ar  @{ -^{>}}^{}  @< 1pt> [l]
	%\ar @{ -^{>}}^{}  @< 1pt> [r]
	%&S_3 \ar  @{ -^{>}}^{}  @< 1pt> [l]
	%\ar @{ -^{>}}^{}  @< 1pt> [r]
	%	&S_4 \ar  @{ -^{>}}^{}  @< 1pt> [l]},
	%\notag
	%\\
	%&\xymatrix{S_1 +S_2\ar[r]^-{}& 2S_2},~~
	%\xymatrix{S_2 +S_3\ar[r]^-{}& 2S_2},\\ \notag
	%&\xymatrix{
	%2S_4\ar @{ -^{>}}^{}  @< 1pt> [r]& S_3+S_4 \ar  @{ -^{>}}^{}  @< 1pt> [l] },~~
	%\xymatrix{3S_3\ar @{ -^{>}}^{}  @< 1pt> [r]& 3S_5 \ar  @{ -^{>}}^{}  @< 1pt> [l]},\\\notag
	%&\qquad \qquad \xymatrix{2S_1\ar[r]^-{}&2S_3\ar[ld]^-{}\\
	%\ar[u]^-{} S_1+S_3 &}.	
	%\end{align}
	Assume this network is constrained by $x_1+x_3+x_5=3, x_1+x_2=2$ and $x_3+x_4=2$. Under the rate constants given below it has an equilibrium $x^*=(1,1,1,1,1)^\top$.
	For convenience of applying \cref{exthm:com-1}, we consider four subnetworks
	
	\emph{(1)} a complex balanced $\mathcal{M}^{(0)}$ on $\{S_1,S_3,S_5\}$ with an equilibrium $(x^*_1,x^*_3,x^*_5)^\top=(1,1,1)^\top$, i.e.,
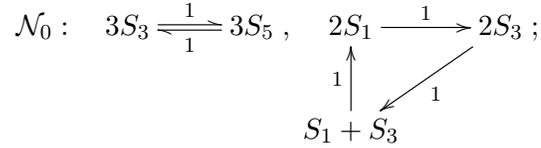
	\begin{align*}
		\mathcal{N}_0:~~\xymatrix{3S_3\ar @{ -^{>}}^{1}  @< 1pt> [r]& 3S_5 \ar  @{ -^{>}}^{1}  @< 1pt> [l]},
		\xymatrix{2S_1\ar[r]^-{1}&2S_3\ar[ld]^-{1}\\
			\ar[u]^-{1} S_1+S_3 &};
	\end{align*}
	
	\emph{(2)} two-species autocatalytic $\mathcal{M}^{(1)}$ and $\mathcal{M}^{(2)}$ on $\{S_1,S_2\}$ and $\{S_2,S_3\}$, which own respective reaction vector equilibrium $(x^*_1,x^*_2)^\top=(1,1)^\top$ and $(x^*_2,x^*_3)^\top=(1,1)^\top$,
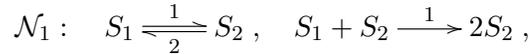
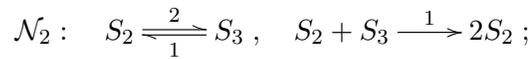
	\begin{align*}
		&\mathcal{N}_1:~~\xymatrix{S_{1} \ar @{ -^{>}}^{1}  @< 1pt> [r]
			& S_{2}\ar  @{ -^{>}}^{2}  @< 1pt> [l]},~~
		\xymatrix{S_1 +S_2\ar[r]^-{1}& 2S_2},
	\end{align*}
	and
	\begin{align*}
		&\mathcal{N}_{2}:~~\xymatrix{ S_{2}
			\ar @{ -^{>}}^{2}  @< 1pt> [r]
			&S_3 \ar  @{ -^{>}}^{1}  @< 1pt> [l]},~~
		\xymatrix{S_2 +S_3\ar[r]^-{1}& 2S_2};
	\end{align*}
	
	\emph{(3)} one $1$-dimensional two-species $\mathcal{M}^{(3)}$ on $\{S_3,S_4\}$, which has a reaction vector balanced equilibrium $(x^*_3,x^*_4)^\top=(1,1)^\top$,
	\begin{align*}
		&\mathcal{N}_3:~~\xymatrix{S_3\ar @{ -^{>}}^{2}  @< 1pt> [r]& S_4 \ar  @{ -^{>}}^{3}  @< 1pt> [l]},~~
		\xymatrix{2S_4\ar @{ -^{>}}^-{1}  @< 1pt> [r]& S_3+S_4 \ar  @{ -^{>}}^-{2}  @< 1pt> [l]}.
	\end{align*}
	
	From \cref{decomposition} and \cref{pro:eq1}, the network is decomposable, and $x^*=(1,1,1,1,1)^\top$ is a generalized equilibrium in $\mathcal{M}$. Further, we have $E_1=\{S_1\}$, $E_2=\{S_3\}$, $E_3=\{S_3\}$,
	$\mathcal{S}^{(1)}\bigcap \mathcal{S}^{(3)}=\emptyset$,
	$\mathcal{S}^{(2)}\bigcap \mathcal{S}^{(3)}=\{S_3\}\in E_2$ and $\mathcal{S}^{(1)}\bigcap \mathcal{S}^{(2)}=\{S_2\}\notin\bigcup^{3}_{i=1}{E}_i$. Finally, we check the other needed conditions in \cref{exthm:com-1}.
	
	\begin{enumerate}[itemindent=3em]
		    \item [$\bullet$] condition $(1)$ in \cref{exthm:com-1} holds since all the reactions in $\mathcal{N}_1$ and $\mathcal{N}_2$ satisfy condition $(2)$ in \cref{thm:com-tw}. Although
		    $\mathcal{N}_3$ is a two-species CRN, it does not meet condition $(2)$ in \cref{thm:com-tw}.
	
		\item [$\bullet$] condition $(2)$ in \cref{exthm:com-1} holds: condition $(1)$ in \cref{thm:com-tw} is true in $\mathcal{M}^{(1)}$, since $\omega_1=(-1,1)$, for $S_2\stackrel{2}{\longrightarrow}S_1$ in $R_{\omega_1}$, $v_{12}^{(1)}-a^{(1)}=0-1=\omega_{11}$, and in $\mathcal{M}^{(2)}$ similar analysis can be made to prove true;
		 condition $(3)$ in \cref{thm:com-tw} holds since $\mathcal{M}^{(1)}$ and $\mathcal{M}^{(2)}$ are two-species autocatalytic MASs and fulfill Eq. \cref{eq:com-tw-con}. Besides, condition $(1)$ in \cref{thm:com-1} is true since in $\mathcal{M}^{(3)}$, $S_3\in E_3$, for $S_3\stackrel{2}{\longrightarrow}S_4$,
		$S_3+S_4\stackrel{2}{\longrightarrow}2S_4$, there are $S_4\stackrel{3}{\longrightarrow}S_3$ and $2S_4\stackrel{1}{\longrightarrow}  S_3+S_4$ ; condition $(2)$ in \cref{thm:com-1} holds in $\mathcal{M}^{(3)}$: $L_{\omega_3}=\{l:v'^{(3)}_{\cdot l}- v^{(3)}_{\cdot l}=(1, -1)\}$ and $R_{\omega_p}=\{l:v'^{(3)}_{\cdot l}- v^{(3)}_{\cdot l}=(-1,1)\}$, so it is not difficult to compute that
		$$\tilde{u}_3(x_4)=\frac{2+2x_4}{3x_4+x^2_4},$$
		and
		$$w_3 \nabla \tilde{u}_3|_{x^*_4=1}=\frac{2x^{*^2}_4+4x^*_4+6}{(3x^*_4+x^{^*2}_4)^2}=\frac{3}{4}>0,$$ which means that $\mathcal{M}^{(3)}$ gratifies Eq. \cref{thm:con-com-1}, too.
	\end{enumerate}
	As a result, the equilibrium $x^*$ of $\mathcal{M}$ is locally asymptotically stable, and the Lyapunov function is given by
	\begin{align}
		f(x)=\sum_{i=1,3,5}(1-x_i+x_i\ln{x_i})
		+\int^{x_2}_{1}\ln \frac{2t}{t+1}dt
		+\int^{x_4}_{1}\ln \frac{3t+t^2}{2+2t}dt.
	\end{align}
\end{example}

\section{Applications: autocatalytic CRNs decomposed for stability}\label{sec5}
In this section, the decomposition strategy is applied to autocatalytic networks for stability analysis.

As stated in \cref{sec:4.1}, autocatalytic CRNs arise abundantly in the biochemical systems, which is deemed as the fundamental to the life process (e.g, self-replications of RNA molecules \cite{hordijk2010autocatalytic}, metabolic pathways). A typical example is an autocatalytic metabolism, such as the Calvin-Benson-Bassham cycle and glyoxylate cycle in central carbon metabolism \cite{barenholz2017design}. Thus the study of stability analysis for autocatalytic CRNs is helpful to understand and engineer the metabolic networks. The previous results in \cref{sec4} provide a feasible solution to address the issue of stability of autocatalytic CRNs.

\begin{property}\label{pro:autoeq}
Assume an autocatalytic $\mathcal{M}=(\mathcal{S,C,R,K})$ admitting an equilibrium $x^*\in \mathbb{R}^n_{>0}$ can be decomposed into several two-species $\mathcal{M}^{(p)}$'s $(p=1,\cdots,\ell)$ according to every pair $(\mathcal{R}_{i,j},\mathcal{R}_{j,i})$. Then $x^*$ is reaction vector balanced in $\mathcal{M}$ if and only if each $x^{*{(p)}}=(x^*_i,x^*_j)\in \mathbb{R}^2_{>0}$, $S_i,S_j\in\mathcal{S}$ is reaction vector balanced in $\mathcal{M}^{(p)}$ .
\end{property}

\begin{proof}
The detailed proof can be found in Appendix 2. \Luref{pro2proof}.
\end{proof}

%\begin{remark}
%Notice that the condition $(3)$ in \cref{auto} guarantees the shared species in subnetworks obtained by decomposition have the same equilibrium values in different reactions pairs, to ensure that the compound of the equilibrium points of the subnetworks becomes the equilibrium of the original one.
%\end{remark}

We thus can present a sufficient condition to suggest the stability of an autocatalytic network through the decomposition strategy.

\begin{theorem}\label{thm:auto}
Given an autocatalytic $\mathcal{M}=(\mathcal{S,C,R,K})$ with an equilibrium $x^*\in\mathbb{R}^n_{>0}$, suppose it can be decomposed into a few two-species autocatalytics $\mathcal{M}^{(p)}$'s $(p=1,\cdots,\ell)$ associated with $(\mathcal{R}_{i,j},\mathcal{R}_{j,i})$. Then $x^*$ is locally asymptotically stable if every $\mathcal{M}^{(p)}$ is reaction vector balanced and gratifies Eqs. \cref{con1} and \cref{con2}.
\end{theorem}

\begin{proof}
Combining \cref{le:auto} and \cref{pro:autoeq}, the result comes immediately.
\end{proof}

We give two examples to exhibit the validity of \cref{thm:auto} and \cref{pro:autoeq}: one is a constructed network, the other is a reduced metabolic autocatalytic network \cite{barenholz2017design}.

\begin{example}
The following autocatalytic network is $4$-dimensional and with deficiency $5$,
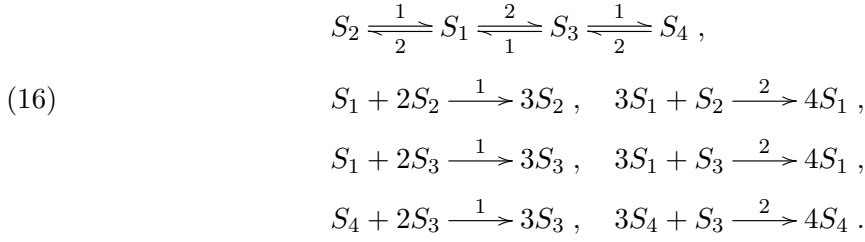
\begin{align}
&\xymatrix{S_{2} \ar @{ -^{>}}^{1}  @< 1pt> [r]
& S_{1} \ar  @{ -^{>}}^{2}  @< 1pt> [l]
\ar @{ -^{>}}^{2}  @< 1pt> [r]
&S_3 \ar  @{ -^{>}}^{1}  @< 1pt> [l]
\ar @{ -^{>}}^{1}  @< 1pt> [r]
&S_4 \ar  @{ -^{>}}^{2}  @< 1pt> [l]},
\notag
\\
&\xymatrix{S_1 +2S_2\ar[r]^-{1}& 3S_2},~~
\xymatrix{3S_1 +S_2\ar[r]^-{2}& 4S_1},\\ \notag
&\xymatrix{
	S_1 +2S_3\ar[r]^-{1}& 3S_3},~~
\xymatrix{3S_1 +S_3\ar[r]^-{2}& 4S_1},\\\notag
&\xymatrix{
	S_4 +2S_3\ar[r]^-{1}& 3S_3},~~
\xymatrix{3S_4 +S_3\ar[r]^-{2}& 4S_4}.\\\notag
\end{align}
Assume it obeys three conservation laws $x_1+x_2=\frac{\sqrt{3}+1}{2}$, $x_2+x_3=\frac{\sqrt{3}+1}{2}$ and $x_3+x_4=\frac{\sqrt{3}+1}{2}$, then it is easy to get that $x^*=(\frac{\sqrt{3}-1}{2},1,1,\frac{\sqrt{3}-1}{2})^\top$ is a reaction vector balanced equilibrium.

We decompose the network into three two-species subsystems $\mathcal{M}^{(p)}|^3_{p=1}$ with $\mathcal{S}^{(1)}=\{S_1,S_2\}$, $\mathcal{S}^{(2)}=\{S_1, S_3\}$,  $\mathcal{S}^{(3)}=\{S_3, S_4\}$. It is also easy to verify that $(x^*_1,x^*_2)^\top=(\frac{\sqrt{3}-1}{2},1)^\top$, $(x^*_1,x^*_3)^\top=(\frac{\sqrt{3}-1}{2},1)^\top$ and $(x^*_3,x^*_4)^\top=(1,\frac{\sqrt{3}-1}{2})^\top$ are reaction vector balanced in $\mathcal{M}^{(1)}$, $\mathcal{M}^{(2)}$ and $\mathcal{M}^{(3)}$, respectively. These results reveal that the MAS share the same equilibrium with that of the subsystems, which is consistent with $\cref{pro:autoeq}$. Finally we can verify that every subsystem $\mathcal{M}^{(p)}$ supports Eqs. \cref{con1} and \cref{con2}, so based on \cref{thm:auto} the MAS is locally asymptotically stable at $(\frac{\sqrt{3}-1}{2},1,1,\frac{\sqrt{3}-1}{2})^\top$ .
\end{example}

\begin{example}
	Consider an $n$-species autcatalytic MAS with $\emph{dim}\mathcal{S}=n-1$ and $\delta=n$,
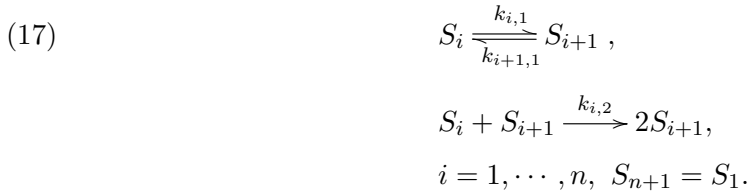
	\begin{align}
	&\xymatrix{S_{i} \ar @{ -^{>}}^{k_{i,1}}  @< 1pt> [r]& S_{i+1} \ar  @{ -^{>}}^{k_{i+1,1}}  @< 1pt> [l]},\\\notag
	&\xymatrix{
		S_i +S_{i+1}\ar[r]^-{k_{i,2}}& 2S_{i+1},} \\ \notag
	&\xymatrix{i=1,\cdots, n, ~S_{n+1}=S_1.}
	\end{align}
The lower part in this network can be viewed as an autocatalytic cycle model \cite{barenholz2017design} without inflow and outflow, which usually appears in metabolic networks.

By selecting $k_{i,1}=k_{i,2}=1, k_{i+1,1}=2$, $i=1,\cdots,n$, and supposing this network to follow the mass conservation laws $x_i+x_{i+1}=2, i=1,\cdots,n-1$, we get $x^*=(1,\cdots,1)^\top$ is a reaction vector balanced equilibrium in this MAS.
	Clearly, this network can be decomposed into $n$ two-species autocatalytic CRNs according to reaction vectors, i.e., $\mathcal{S}^{(1)}=(S_1, S_2),$ $\mathcal{S}^{(2)}=(S_2, S_3),$ $\cdots,$ $\mathcal{S}^{(n)}=(S_{n}, S_1)$, respectively. Further, it is not hard to compute that every subnetwork successively admits a reaction vector balanced equilibrium
	$(1,1)^\top$, which is consistent with \cref{pro:autoeq}. Finally, according to \cref{re:stability} and \cref{thm:auto}, the equilibrium $x^*$ is locally asymptotically stable.
\end{example}

%\begin{corollary}\emph{(The composition of autocatalytic CRNs)}
%For a $\mathcal{M}$ that is compounded of a series of two-species autocatalytic $\mathcal{M}^{(p)}$s, $p=1,\cdots,\ell$, suppose $\mathcal{M}$ possesses a reaction vector balanced equilibrium $x^*$, then $x^*$ is locally asymtotically stable if every $\mathcal{M}^{(p)}$ satisfies Eq. \cref{con1} and Eq. \cref{con2}, and the compounded $\mathcal{M}$ meets the condition $4$ in \cref{auto}.
%\end{corollary}

\section{Conclusion}\label{sec6}
From the above discussion, the conclusions can be reached that it is possible to capture the stability properties of some large-scale CRNs through the proposed network decomposition technique. Some sufficient conditions are presented to suggest stability if the network can be decomposed into a complex balanced subnetwork and a few $1$-dimensional subnetworks (or a few two-species subnetworks) whether the species are shared among subnetworks or not. The results are finally applied to autocatalytic networks for validation.

\section*{Appendix}

\setcounter{equation}{0}
\setcounter{theorem}{0}
\setcounter{subsection}{0}
\setcounter{example}{0}
\gdef\theequation{A.\arabic{equation}}
\gdef\thetheorem{A.\arabic{theorem}}
\gdef\thesubsection{A.\arabic{subsection}}
\gdef\theexample{A.\arabic{example}}
\section*{Appendix 1: CRNs and related stability resutlts}\label{ap1}
In this section, some basic concepts associated with CRNs are reviewed. For those special CRNs with stability property, the available Lyapunov functions are recalled.
\subsection{CRNs}\label{ap2}
Consider a network consisting of $n$ species $S_1, \cdots, S_n$ and $r$ chemical reactions. We denote the $i$th $(i=1,\cdots,r)$ reaction as
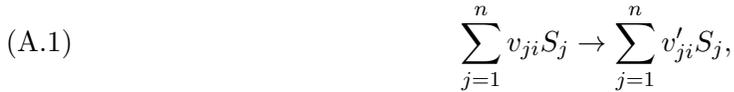
\begin{equation}\label{reaction}
\sum^{n}_{j=1}v_{ji}S_j \rightarrow \sum^{n}_{j=1}v'_{ji}S_j,
\end{equation}
where $v_{ji}, ~v'_{ji}\in\mathbb{Z}_{\geq 0}$ are the complexes of reactant and product, respectively.

We give some definitions related to CRNs \cite{Feinberg1995The} that will be used throughout the paper.

\begin{definition} \emph{(CRN)}. 
A CRN is composed of three finite sets:
	\begin{enumerate}[itemindent=3em]
		\item{a set of species $\mathcal{S}=\{S_1, \cdots, S_n\}$;}
		\item{a set of complexes $\mathcal{C}=\bigcup^r_{i=1}\{v_{\cdot i},v'_{\cdot i}\}$  and the $j$th entry of $v_{\cdot i}$ is the stoichiometric coefficient of $S_j$ in this complex, based on which \cref{reaction} is simply written as $v_{.i}\to v'_{.i}$;}
		\item{a set of reactions $\mathcal{R}=\{v_{\cdot 1}\rightarrow v'_{\cdot 1},\cdots,v_{\cdot r}\rightarrow v'_{\cdot r}\}$, satisfying $\forall~ v_{\cdot i} \in \mathcal{C}, v_{\cdot i}\rightarrow v_{\cdot i}\notin \mathcal{R}$ but $\exists~ v'_{\cdot i}$, s.t. $v_{\cdot i}\rightarrow v'_{\cdot i} \in\mathcal{R}$ or $v'_{\cdot i}\rightarrow v_{\cdot i}\in\mathcal{R}$.}
	\end{enumerate}
The triple $(\mathcal{S,C,R})$ or $\mathcal{N}\triangleq(\mathcal{S,C,R})$ is usually used to indicate a CRN.
\end{definition}

\begin{definition}\emph{(weakly reversible and reversible CRN)}
A CRN $(\mathcal{S,C,R})$ is weakly reversible if for each reaction $v_{\cdot i}\rightarrow v'_{\cdot i} \in\mathcal{R}$, there exists a chain of reactions, which starts from $v'_{\cdot i}$ and ends with $v_{\cdot i}$, i.e., $v'_{\cdot i}\rightarrow v_{\cdot i_1} \in\mathcal{R}$, $v_{\cdot i_1}\rightarrow v_{\cdot i_2} \in\mathcal{R}, \cdots, v_{\cdot i_m}\rightarrow v_{\cdot i} \in\mathcal{R}$. In particular, if $\forall v_{\cdot i}\rightarrow v'_{\cdot i} \in\mathcal{R}$, there is $v'_{\cdot i}\rightarrow v_{\cdot i} \in\mathcal{R}$, we say that this network is reversible.
\end{definition}

\begin{definition}\emph{(stoichiometric subspace)}.
	For a CRN $(\mathcal{S,C,R})$, the linear subspace $\mathscr{S}\triangleq \emph{span}\{v'_{\cdot 1}-v_{\cdot 1},\cdots,v'_{\cdot r}-v_{\cdot r}\}$ is called the stoichiometric subspace of the network, and its dimension \emph{dim}$\mathscr{S}$ is called the dimension of the network.
\end{definition}

\begin{definition}\emph{(stoichiometric compatibility class).}
Given a CRN $(\mathcal{S,C,R})$,
for $x_0\in\mathbb{R}_{\geq0}^{n}$, we say the sets $\mathscr{S}(x_0)\triangleq\{x_0+\xi\mid\xi\in\mathscr{S}\}$, $\bar{\mathscr{S}}^+(x_0)\triangleq\mathscr{S}(x_0)\bigcap \mathbbold{R}^n_{\geq 0}$ and $\mathscr{S}^+(x_0) \triangleq\mathscr{S}(x_0)\bigcap \mathbbold{R}^n_{>0}$ are the stoichiometric compatibility class, nonnegative and positive stoichiometric compatibility class of $x_0$, respectively.
\end{definition}

\begin{definition}\emph{(deficiency)}
For a CRN $(\mathcal{S,C,R})$, $$\delta=|\mathcal{C}|-l-\emph{dim}\mathscr{S}$$ is defined as the deficiency of the network with
$|\mathcal{C}|$ and $l$ represent the numbers of complexes and of linkage classes in the network, respectively.
\end{definition}

When a CRN obeys the mass action law, the reaction rate may be measured according to the power law with respect to the species concentrations, e.g., for the $i$th reaction $v_{\cdot i}\rightarrow v'_{\cdot i}$ the reaction rate is
$$\Xi_i(x)= k_{i} x^{v_{\cdot i}}\triangleq k_i\prod_{j=1}^{d}x_{j}^{v_{ji}},$$
where $k_i \in \mathbb{R}_{>0}$ is the reaction rate constant, and $x\in \mathbb{R}^n_{\geq 0}$ represents the concentration vector of the species $\mathcal{S}_i$ with each element $x_i$.

\begin{definition}\emph{(MAS).}
	A CRN $(\mathcal{S,C,R})$ with mass-action kinetics is called an MAS, represented by the quadruple $\mathcal{M}\triangleq(\mathcal{S,C,R,K})$, where $\mathcal{K}=\{k_1,\cdots,k_r\}$ is the set of reaction rate constants.
\end{definition}

In the context, we are mainly concerned with MASs. Denote the stoichiometric matrix of $\mathcal{N}=(\mathcal{S,C,R})$ by $\Gamma\in\mathbb{Z}_{n\times r}$ with the $i$th column $\Gamma _{\cdot i}=v'_{\cdot i}-v_{\cdot i}$, termed reaction vector, and the $r$-dimensional non-negative vector function of reaction rate by $\Xi(x)$ with each component $\Xi_{i}(x)$.
Then the dynamics of $\mathcal{M}=(\mathcal{S,C,R,K})$ follows
\begin{equation}\label{eq:mas}
\frac{\mathrm{d}x}{\mathrm{d}t}=\Gamma \Xi(x),~~~x\in \mathbb{R}_{\geq0}^{n}.
\end{equation}

\begin{definition}\emph{(equilibrium)}
For an $\mathcal{M}=(\mathcal{S,C,R,K})$, a positive point $x^{*}\in\mathbb{R}_{>0}^{n}$ is an equilibrium in $\mathcal{M}$ if it satisfies $\Gamma \Xi(x^{*})=0$. An MAS that admits an equilibrium is a balanced MAS.
\end{definition}

Following the notion of balanced CRN, we introduce a new concept, termed as generalized balancing, which first emerged in \cite{hoessly2019stationary} in a stochastic version. In this paper, we extended this notion to the deterministic case.
\begin{definition}\emph{(generalized balancing)}
\label{def:generalized balanced}
	We say an $\mathcal{M}=(\mathcal{S,C,R,K})$ is generalized balanced at $x^{*}\in\mathbb{R}_{>0}^{n}$ if there exists a set of tuples of subsets of reaction set $\mathcal{R}$, denoted as $\{(L_i,R_i)_{i\in A}\} $ with
	$$\bigcup_{i\in A} L_i=\bigcup_{i\in A} R_i=\mathcal{R}$$
	such that $\forall i \in A$
	\begin{align}
	\sum_{v_{\cdot i}\rightarrow v_{\cdot i}^{'} \in L_i}k_{i}(x^{*})^{v_{\cdot i}}=
	\sum_{v_{\cdot i}\rightarrow v_{\cdot i}^{'} \in R_i}k_{i}(x^{*})^{v_{\cdot i}}
	\end{align}
holds true.	
\end{definition}

\begin{remark}
\cref{def:generalized balanced} declares that generalized balancing encompasses the following cases,
	\begin{itemize}[itemindent=3em]
			\item detailed balancing with the set $A$ induced by every reaction, i.e., $A=\{i~|~(v_{\cdot i}\longrightarrow v'_{\cdot i}, v'_{\cdot i}\longrightarrow v_{\cdot i})_{v_{\cdot i}\longrightarrow v'_{\cdot i} \in \mathcal{R}}\}$,
		\begin{align*}
		L_i=\left\lbrace v_{\cdot i}\rightarrow v'_{\cdot i}\in \mathcal{R}\right\rbrace,~R_i=\left\lbrace v'_{\cdot i}\rightarrow v_{\cdot i}\in \mathcal{R}\right\rbrace,
		\end{align*}
		\item complex balancing with the set $A$ induced by any $z \in \mathcal{C}$, i.e., $A=\{z~|~z\in\mathcal{C}\}$,
		\begin{align*}
		L_z=\left\lbrace v_{\cdot i}\rightarrow v'_{\cdot i} \in \mathcal{R}\mid v_{\cdot i}=z \right\rbrace,\\
		R_z=\left\lbrace v_{\cdot i}\rightarrow v'_{\cdot i} \in \mathcal{R}\mid v'_{\cdot i}=z \right\rbrace,
		\end{align*}
		\item  reaction vector balancing \cite{Cappelletti2018Graphically} with the set $A$ induced by every given reaction vector $ \eta \in \mathbbold{Z}^n$, i.e., $A=\{\eta~|~\eta:=v'_{\cdot i}-v_{\cdot i}, \forall i\}$,
		\begin{align*}
		&L_{\eta}=\left\lbrace v_{\cdot i}\rightarrow v'_{\cdot i} \in \mathcal{R}\mid v'_{\cdot i}-v_{\cdot i}=\eta \right\rbrace,\\
		&R_{\eta}=\left\lbrace v_{\cdot i}\rightarrow v'_{\cdot i} \in \mathcal{R}\mid v'_{\cdot i}-v_{\cdot i}=-\eta \right\rbrace.
		\end{align*}
		\item any combinatorial forms of the above three types of balancing and other possibilities.
 	\end{itemize}
\end{remark}

Note that a generalized balanced MAS must be a balanced MAS. The former two classes of balancing are often encountered in the literature with detailed balancing request the structure of the network to be reversible while complex balanced network to be weakly reversible. It is obvious that a detailed balanced network must be a complex balanced one, and not vice versa. Let us consider the following example to illustrate the concept of reaction vector balancing.

\begin{example}\label{eg:Aurora}
Consider the network describing the autophosphorylation of Aurora B kinase, given by
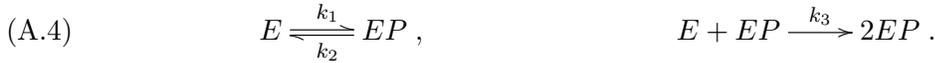
\begin{align}
	&\xymatrix{E\ar @{ -^{>}}^{k_1}  @< 1pt> [r]
		&EP \ar  @{ -^{>}}^{k_2}  @< 1pt> [l]},
	&\xymatrix{E+EP \ar[r]^-{k_3}	& 2EP}.
\end{align}
where both trans- and cis-autophosphorylation reactions happens, and $E, EP$ are proteins. There are $4$ complexes, $2$ linkage classes, and \emph{dim}$\mathcal{S}=1$, which implies that the deficiency $\delta=4-2-1=1$. Let $S_1=E, S_2=EP$, there are two reaction vectors $(-1,1)^\top$, $(1,-1)^\top$ in the network. If this system has a positive reaction vector balanced equilibrium $(x^*_1,x^*_2)$, then it should fulfill the following equality $k_1x^*_1+k_3x^*_1x^*_2=k_2x^*_2$, i.e., $(x^*_1,x^*_2)=(c, \frac{ck_1}{k_2-ck_3})$ with $c<\frac{k_2}{k_3}$ a positive constant.
\end{example}

\subsection{Some known Lyapunov functions}
As stated in the Introduction, for some special CRNs, such as those with weakly reversible structure and those with dimension $1$, their stability (under some moderate conditions) has been well studied and the corresponding Lyapunov functions are also proposed. For the former under complex balancing condition, the well-known pseudo-Helmholtz free energy function \cite{HornJackson1972General} is an available one, which takes
\begin{align}\label{eq:Helmholtz}
		G(x)=\sum^{n}_{j=1}\left(x^*_j-x_j-x_j\ln{\frac{x^*_j}{x_j}}\right),~~ x\in\mathbb{R}^n_{>0}.
	\end{align}
For any $1$-dimensional $\mathcal{M}=(\mathcal{S,C,R,K})$, an available Lyapunov function \cite{Fang2015Lyapunov} is given by
	\begin{align}\label{sub1Lya}
		f(x)=\int^{\gamma(x)}_0 \ln \tilde{u}(y^{\dag}(x)+\alpha \omega)\dd \alpha, ~~ x\in\mathbb{R}^n_{>0}
	\end{align}	
with $\tilde{u}$ satisfy $h(x,\tilde{u})=0$ and $h(x,u)$ defined by
   \begin{align}\label{eq:sub1Lya}
	h(x,u)=\sum_{\{i|\beta_{i} >0\}}({k}_i {x}^{{v}_{\cdot i}})\bigg(\sum^{\beta_{i}-1}_{j=0} u^j \bigg)
	+\sum_{\{i|\beta_{i} <0\}}({k}_i  {x}^{{v}_{\cdot i}})\bigg(-\sum^{-1}_{j=\beta_{i}} {u}^j \bigg),
   \end{align}
where $\omega \in \mathbb{R}^{n}\setminus \{\mathbbold{0}_{n}\}$ is a set of bases of ${\mathscr{S}}$, and $\beta_{i}\in \mathbb{Z}\setminus\{0\}$ satisfies ${v}'_{\cdot i}-{v}_{\cdot i}=\beta_{i} \omega, ~i=1,\cdots,r$, and moreover,
 $\gamma \in\mathscr{C}^2(\mathbb{R}^n_{>0};\mathbb{R}_{>0})$, $y^{\dag}\in\mathscr{C}^2(\mathbb{R}^n_{>0};\mathbb{R}^n_{>0})$ are constrained by $x=y^{\dag}(x)+\gamma(x)\omega$ and
$\gamma(x+b \omega)=\gamma(x)+b$ $,~ \forall b\in \mathbb{R}$, respectively.

\section*{Appendix 2: Detailed proofs}
In the appendix, we provide detailed proofs for some results appearing in \cref{sec3}, \cref{sec4} and \cref{sec5}.
\subsection*{B. The proof of \cref{thm:com-1}}\label{Thm3.3proof}
\setcounter{equation}{0}
\gdef\theequation{B.\arabic{equation}}

Denote the state of $\mathcal{M}^{(p)}$ by $x^{(p)}\in \mathbb{R}^{n_p}_{>0}$. Since $\forall p\neq q\in\{0,1,...,\ell\}$,~$E_p:= \mathcal{S}^{(0)}\bigcap\mathcal{S}^{(p)}\neq\emptyset$, from \cref{decomposition}, we can get the state of $\mathcal{M}$: $x=\bigg(\bigotimes_{S_i\in \mathcal{S}^{(0)}}x^{(0)}_i,\bigotimes^{\ell}_{p=1}\bigotimes_{S_j\in \mathcal{S}^{(p)}\backslash E_p}x^{(p)}_j\bigg)^\top$ where $ x_i,x_j$ represent the species $S_i,S_j$.
Meanwhile, according to the decomposition, the dynamics of the $\mathcal{M}$ can be represented as
	\begin{align}\label{eq:dynamicCB-sub1}
		\dot{x}=\sum^{\ell}_{p=0} \sum^{r_{p}}_{i=1}k^{(p)}_i x^{v^{(p)}_{\cdot i}}
		\bigg(v'^{(p)}_{\cdot i}-v^{(p)}_{\cdot i}\bigg),%\tag{A.1}
	\end{align}
where $r_p$ is the number of reactions in $\mathcal{M}^{(p)}$.

	Since every $\mathcal{M}^{(p)}$ for $p=1,\cdots,\ell$ is reaction vector balanced, plus condition $(2)$ listed in \cref{thm:com-1}, without loss of generality,
	the reactions in $\mathcal{M}^{(p)}$ are supposed to be classified into two groups, which respectively are denoted by $L_{\omega_p}=\{l:v'^{(p)}_{\cdot l}- v^{(p)}_{\cdot l}=(\bigotimes_{S_i\in E_p}{1}, \omega^\top_p)^\top\}$ and $R_{\omega_p}=\{l:v'^{(p)}_{\cdot l}- v^{(p)}_{\cdot l}=-(\bigotimes_{S_i\in E_p}{1}, \omega^\top_p)^\top\}$. Since $\mathcal{M}^{(p)}$ is $1$-dimensional, we can construct a Lyapunov candidate for $\mathcal{M}^{(p)}$ which is a special form of \cref{sub1Lya},
	\begin{align*}\label{eq:Lya2}
		f^{(p)}(x^{(p)})&=\sum_{S_i\in E_p}
		\bigg(x_i^*{^{(p)}}-x^{(p)}_i-x_i\ln\frac{x_i^*{^{(p)}}}{x^{(p)}_i}\bigg)+
		\int^{\gamma{(\tilde{x}^{(p)})}}_{0}\ln\tilde{u}_p(y^{\dag}(\tilde{x}^{(p)})+\omega_p t) dt,
	\end{align*}
	where $\tilde{x}^{(p)}=(\bigotimes_{S_j\in \mathcal{S}^{(p)}\backslash E_p} x^{(p)}_j)^\top$, represents the components that are not involved in the complex balanced $\mathcal{M}^{(0)}$,
	$\gamma \in\mathscr{C}^2(\mathbb{R}^{n_p-|E_p|}_{>0};\mathbb{R}_{>0})$ , $y^{\dag}\in\mathscr{C}^2(\mathbb{R}^{n_p-|E_p|}_{>0};\mathbb{R}^{n_p-|E_p|}_{>0})$ and $\omega_p$ share the same meanings with \cref{sub1Lya}.
	
	Combing the Lyapunov function \cref{eq:Helmholtz} for complex balanced $\mathcal{M}^{(0)}$, which is represented as
    \begin{align}%\label{eq:Lya3}
    	f^{(0)}(x^{(0)})&=\sum_{S_i\in \mathcal{S}^{(0)}}
    		\bigg(x_i^*{^{(0)}}-x^{(0)}_i-x^{(0)}_i\ln\frac{x_i^*{^{(0)}}}{x^{(0)}_i}\bigg), %\tag{A.2}
    \end{align} it
	is easy to derive the following function
	\begin{align}%\label{eq:Lya4}
		f(x)=
		f^{(0)}(x^{(0)})+\sum^{\ell}_{p=1}\int^{\gamma{(\tilde{x}^{(p)})}}_{0}\ln\tilde{u}_p(y^{\dag}(\tilde{x}^{(p)})+\omega_p t) dt. %\tag{A.3}
	\end{align}

	The next step is to verify that $f(x)$ is able to be a Lyapunov function for $\mathcal{M}$. Because of the continuity of $\tilde{u}_p(\tilde{x}^{(p)})$ and \cref{thm:con-com-1}, we know that, for $p=1,\cdots,\ell$, there must exist a neighborhood of $\tilde{x}^{*^{(p)}}$, denoted by $\mathcal{N}(\tilde{x}^{*^{(p)}})$, s.t. $\forall \tilde{x}^{(p)} \in \mathcal{N}(\tilde{x}^{*^{(p)}})$, it holds
	\begin{align*}
		\omega^\top_p \nabla \tilde{u}_p(\tilde{x}^{(p)})>0.
	\end{align*}
	Since $\tilde{u}_p(\tilde{x}^{(p)})=\exp\{\omega^\top_p
		\nabla f^{(p)}(\tilde{x}^{(p)})\}$, we have
	$$\nabla \tilde{u}_p(\tilde{x}^{(p)})=\tilde{u}_p\nabla^2 f^{(p)}(\tilde{x}^{(p)})\omega_p. $$
	From \cref{pro:eq1}, $x^*$ is a generalized equilibrium point in $\mathcal{M}$,
	then $\forall x \in \mathcal{D}(x^*)$ with $$\mathcal{D}(x^*)\triangleq\mathscr{S}^{+}(x^*)\bigcap \big(\mathbb{R}^{n_0}_{>0}\bigotimes^{\ell}_{p=1}\mathcal{N}(\tilde{x}^{*^{(p)}})\big),$$ and $\forall \mu \in\mathscr{S}$, there is
	\begin{align*}
		&~\quad
		\mu^\top\nabla^2f(x)\mu
		\notag\\
		&=
		\underbrace{\mu^{(0)^\top}
			\text{diag}
			(\bigotimes_{S_i\in\mathcal{S}^{(0)}}1/{x^{*^{(0)}}_i})
			\mu^{(0)}}_{g_1\geq 0}
		+\sum^{\ell}_{p=1} \mu^{(p)^\top} \nabla^2 f^{(p)}(\tilde{x}^{(p)})\mu^{(p)}
		\notag\\
		&={g_1}
		+\sum^{\ell}_{p=1}
		\frac{\mu^{(p)^\top}\mu^{(p)}}{\omega^\top_p\omega_p}
		\omega^\top_p \nabla^2 f^{(p)}(\tilde{x}^{(p)})\omega_p
		\notag\\
		&={g_1}
		+\sum^{\ell}_{p=1}
		\frac{\mu^{(p)^\top}\mu^{(p)}}{\omega^\top_p\omega_p}
		\frac{\omega^\top_p \nabla \tilde{u}_{p}(\tilde{x}^{(p)})}
		{\tilde{u}_{p}}
		\notag\\
		&\geq0,%\tag{A.4}
	\end{align*}
	with the equality holding if and only if $\mu=\mathbbold{0}_n$.
	Thus $f(x)$ is strictly convex in this region. Because of the fact $\nabla f(x^*)=\bigg(\big(\text{Ln}\frac{x^{(0)}}{{x^*}^{(0)}}\big)^\top,\bigotimes^{\ell}_{p=1}\frac{\omega^\top_p }{\omega^\top_p\omega_p}\ln{\tilde{u}(\tilde{x}^{(p)})}\bigg)^\top|_{x=x^*}=\mathbbold{0}^\top_n$, we obtain $f(x)\geq f(x^*)=0$ for $x\in \mathcal{D}(x^*)$.
	
	Then we need to prove $\dot {f}(x)\leq0$.
	Since $\frac{\partial f(x)}{\partial x^{(p)}}=\frac{\partial f^{(p)}(x^{(p)})}{\partial x^{(p)}}$, $p=1,\cdots,\ell$, we have
	\begin{align*}
	\dot {f}(x)&=\nabla^\top f(x) \dot {x}
	\notag
	\\
	&=\underbrace{\bigg(\text{Ln}\frac{x^{(0)}}{{x^*}^{(0)}}\bigg)^\top \cdot \sum^{r_{0}}_{i=1}k^{(0)}_i x^{v^{(0)}_{\cdot i}}
		\bigg(v'^{(0)}_{\cdot i}-v^{(0)}_{\cdot i}\bigg)}_{\dot f^{(0)}(x^{(0)})}+
	\sum^{\ell}_{p=1}\underbrace{\sum^{r_{p}}_{i=1}k^{(p)}_i x^{v^{(p)}_{\cdot i}}\beta_i \omega^\top_p
			 \nabla f^{(p)}}_{\dot f^{(p)}(x^{(p)})},%\tag{A.5}
	\end{align*}
where the second part satisfies
\begin{align*}
	\dot f^{(p)}(x^{(p)})\leq \sum^{r_{p}}_{i=1}k^{(p)}_i x^{v^{(p)}_{\cdot i}}(\exp\{\beta_i\omega^\top_p
		\nabla f^{(p)}(x^{(p)})\}-1)=0 %\tag{A.6}
	\end{align*}
with the equality holding only if $\omega^\top_p
\nabla f^{(p)}({x}^{(p)})=0$, that is $\prod_{S_i\in E_p}\frac{x^{(p)}_i}{{x^*_i}^{(p)}}\tilde{u}(\tilde{x}^{(p)})=1$, which means ${x}^{(p)}={x}^*{^{(p)}}$. The remaining proof is similar to that of \cref{thm:disjoint} and we omit it here. $\Box$

\subsection*{C. The proof of \cref{thm:com-tw}}\label{Thm16proof}
\setcounter{equation}{0}
\gdef\theequation{C.\arabic{equation}}

As we know, the complex balanced $\mathcal{M}^{(0)}$ has a Lyapunov function in the form of
	\begin{align*}
		f^{(0)}(x^{(0)})=\sum_{S_i\in \mathcal{S}^{(0)}}
		(x^*_i-x_i-x_i\ln\frac{x^*_i}{x_i})
	\end{align*}
	to render the asymptotic stability of $x^{(0)}$.
Besides, $\forall p=1,\cdots,\ell$, there is a candidate Lyapunov function for $\mathcal{M}^{(p)}$, which is given by
	\begin{align*}
		f^{(p)}(x^{(p)})=
		(x^*_i-x_i-x_i\ln\frac{x^*_i}{x_i})
		+\int^{x_j}_{x^*_j}w^{-1}_{jp}\ln \frac{t^{b^{(p)}}}{c^{(p)}_{ij}\sum_{ L_{\omega_p}}k^{(p)}_lt^{v_{jl}^{(p)}}}dt,
	\end{align*}
	where $S_i\in E_p, S_j\in \mathcal{S}^{(p)}\backslash E_p $, and
	$c^{(p)}_{ij}=\frac{{x^*_i}^{a^{(p)}}}{\sum_{ R_{\omega_p}}k^{(p)}_lx_i^{*^{v_{il}^{(p)}}}}=\frac{{x^*_j}^{b^{(p)}}}{\sum_{ L_{\omega_p}}k^{(p)}_lx_j^{*^{v_{jl}^{(p)}}}}$.
	
	Integrating the above information, we can construct the following function
	\begin{align}\label{eq:Lya1}
		f(x)=f^{(0)}(x^{(0)})+
		\sum^{\ell}_{p=1}\int^{x_j}_{x^*_j}w^{-1}_{jp}\ln \frac{t^{b^{(p)}}}{c^{(p)}_{ij}\sum_{ L_{\omega_p}}k^{(p)}_lt^{v_{jl}^{(p)}}}dt.%\tag{B.1}
	\end{align}

	Ulteriorly, we aim to state $f(x)$ is localy positive definite. By computing the Hessian matrix of $f(x)$, we derive
	\begin{align*}
		\nabla^2 f(x)=
		\text{diag}
		\bigg(
		\bigotimes_{S_i\in\mathcal{S}^{(0)}} x^{*^{-1}}_i,
		\bigotimes_{S_j\in\mathcal{S}\backslash \mathcal{S}^{(0)}}
		\frac{w^{-1}_{jp}\sum_{L_{\omega_p}} k^{(p)}_l(b^{(p)}-v_{jl}^{(p)})x^{v_{jl{(p)}}-1}_{j}}{\sum_{L_{\omega_p}} k^{(p)}_lx^{v_{jl}^{(p)}}_{j}}
		\bigg),
	\end{align*}
	where $\bigotimes$ stands for the Cartesian product.
	Since $w^{-1}_{jp}\sum_{L_{\omega_p}} k^{(p)}_l(b^{(p)}-v_{jl}^{(p)})x^{{v_{jl}^{(p)}-1}}_{j}$ is continuous, by the third condition in \cref{thm:com-tw}, there is a neighborhood of $x^*_j$ for $S_j\in\mathcal{S}\backslash \mathcal{S}^{(0)}$, denoted by $\mathcal{N}(x^*_j)$, such that
	$\forall x_j \in \mathcal{N}(x^*_j)$, we have
		$w^{-1}_{jp}\sum_{L_{\omega_p}} k^{(p)}_l(b^{(p)}-v_{jl}^{(p)})x^{*^{v_{jl}^{(p)}-1}}_{j}>0.$
	Thus,
	$\forall x\in \big(\mathbb{R}^{n_0}_{>0} \bigotimes_{S_j\in\mathcal{S}\backslash \mathcal{S}^{(0)}}\mathcal{N}(x^*_j)\big)
	\bigcap \mathscr{S}^+(x^*)$, we get $\nabla^2 f(x)>0$. Moreover, it is easy to obtain $\nabla f(x^*)=\left(\bigotimes_{S_i\in\mathcal{S}^{(0)}} \ln \frac{x_i}{x^*_i}, 	\bigotimes_{S_j\in\mathcal{S}\backslash \mathcal{S}^{(0)}}w^{-1}_{jp}\ln \frac{{x^*_j}^{b^{(p)}}}{c^{(p)}_{ij}\sum_{ L_{\omega_p}}k^{(p)}_l{x^*_j}{^{v_{jl}^{(p)}}}}\right)^\top=\mathbbold{0}_{n}^\top$, which implies $f(x)$ is lower bounded by $f(x^*)=0$.
	Then we continue to certify $\dot{f}(x)\leq0$. Here, we use the similar form of \cref{eq:dynamicCB-sub1} to represent the dynamics of $\mathcal{M}$ although the parameters and complexes different, and derive
 \begin{align*}
 	\dot {f}(x)
 	&=\underbrace{\bigg(\bigotimes_{S_i\in\mathcal{S}^{(0)}}\ln\frac{x_i}{x^*_i}\bigg)^\top \sum^{r_{0}}_{i=1}k^{(0)}_i x^{v^{(0)}_{\cdot i}}
 		\bigg(v'^{(0)}_{\cdot i}-v^{(0)}_{\cdot i}\bigg)}_{\dot f^{(0)}(x^{(0)})}+
 	\sum^{\ell}_{p=1}\underbrace{\sum^{r_{p}}_{i=1}k^{(p)}_i x^{v^{(p)}_{\cdot i}}
 		\bigg(v'^{(p)}_{\cdot i}-v^{(p)}_{\cdot i}\bigg)^\top\nabla f^{(p)}}_{\dot f^{(p)}(x^{(p)})}.%\tag{B.4}
 \end{align*}
It is obvious that $\dot f^{(0)}(x^{(0)})\leq0$ with the equality holding only at $x_i=x^*_i$ for $S_i\in\mathcal{S}^{(0)}$, since $\mathcal{M}^{(0)}$ is complex balanced. Further, since every $\mathcal{M}^{(p)}$, $p=1,\cdots,\ell$, is reaction vector balanced, according to \cref{le:tw}, there is $\dot f^{(p)}(x^{(p)})\leq0$ where the equality holds only if $x^{(p)}={x^*}^{(p)}$, which directly follows \cref{eq:dissipative}. Consequently it states that $\dot f(x)\leq 0$ with $\dot f(x)=0$ if and only if $x=x^*$. 
$\Box$

\subsection*{D. The proof of \cref{pro:autoeq}}\label{pro2proof}
\setcounter{equation}{0}
\gdef\theequation{D.\arabic{equation}}
	In terms of the reactions pair $(\mathcal{R}_{i,j},\mathcal{R}_{j,i})$, the dynamics of the considered MAS can be represented by
	\begin{align}\label{autodynamic}
		\dot{x}_i=\sum_{S_j\in\mathcal{S}}
		\bigg(
		\sum_{ \mathcal{R}_{j,i}}k_{j,\alpha_i}x_j x^{\alpha_i-1}_i-
		\sum_{ \mathcal{R}_{i,j}}k_{i,\alpha_j}x_i x^{\alpha_j-1}_j
		\bigg)%\tag{C.1}
	\end{align}
	where $i=1,\cdots,n$. Then we explain the sufficiency.
	Assume every $\mathcal{M}^{(p)}$ has a reaction vector balanced equilibirum $x^{*^{(p)}}=(x^*_i,x^*_j)\in\mathbb{R}^2_{>0}$, for
	$p=1,\cdots,\ell$, $i,j\in \{1,\cdots,n\}$.
	Based on condition $(3)$ in \cref{auto}, we might as well assume $i$ is fixed and $j=1,\cdots,m$, thus, there is (1)
	$c_jk_{1,\alpha_i}=k_{j,\alpha_i}, \forall j=1,\cdots,m_i$, for all reaction rates in $\mathcal{R}_{ji}$, (2) $c_jk_{i,\alpha_1}=k_{i,\alpha_j}, \forall j=1,\cdots,m_i$, for all reaction rates in $\mathcal{R}_{ij}$. So,
	Eq. \cref{autodynamic} can be transformed into
	\begin{align}
		\dot{x}_i=\sum^{m_i}_{j=1}c_j
		\bigg(
		\sum_{ \mathcal{R}_{j,i}}k_{1,\alpha_i}x_j x^{\alpha_i-1}_i-
		\sum_{ \mathcal{R}_{i,j}}k_{i,\alpha_1}x_i x^{\alpha_j-1}_j
		\bigg),%\tag{C.2}
	\end{align}
	which implies each $(\mathcal{R}_{i,j},\mathcal{R}_{j,i})$ has a reaction vector balanced equilibrium $(x^*_i, x^*_j), j=1,\cdots,m$. Namely, we can know that for $(\mathcal{R}_{i,j},\mathcal{R}_{j,i})$ with different $j$s and fixed $i$, they share the same concentration $x^*_i$ at their respective equilibria. As a result, the compounded concentration $x^*=(x^*_1,\cdots,x^*_n)$ from all $x^{*{(p)}}$ is a reaction vector balanced equilibirum of $\mathcal{M}$.
	In reverse, the necessity is obvious. $\Box$
%\begin{ack}                               % Place acknowledgements
%Partially supported by the Roman Senate.  % here.
%\end{ack}

%\bibliographystyle{unsrt}        % Include this if you use bibtex
%\bibliography{references}           % and a bib file to produce the
                                 % bibliography (preferred). The
                                 % correct style is generated by
                                 % Elsevier at the time of printing.

%\begin{thebibliography}{99}     % Otherwise use the
                                 % thebibliography environment.
                                 % Insert the full references here.
                                 % See a recent issue of Automatica
                                 % for the style.
  %\bibitem[Heritage, 1992]{Heritage:92}
    % Heritage (1992) {\it The American Heritage.
    % Dictionary of the American Language.}
     %Houghton Mifflin Company.
 % \bibitem[Able, 1956]{Abl:56}
     %B.~C.~Able (1956). Nucleic acid content of macroscope.
     %{\it Nature 2}, 7--9.
  %\bibitem[Able {\em et al.}, 1954]{AbTaRu:54}
    %B.~C. Able, R.~A. Tagg, and M.~Rush (1954).
   % Enzyme-catalyzed cellular transanimations.
     %In A.~F.~Round, editor,
     %{\it Advances in Enzymology Vol. 2} (125--247).
     %New York, Academic Press.
  %\bibitem[R.~Keohane, 1958]{Keo:58}
     %R.~Keohane (1958).
    % {\it Power and Interdependence:
    % World Politics in Transition.}
    % Boston, Little, Brown \& Co.
  %\bibitem[Powers, 1985]{Pow:85}
    % T.~Powers (1985).
    % Is there a way out?
     %{\it Harpers, June 1985}, 35--47.

%\end{thebibliography}

\bibliographystyle{siamplain}
\bibliography{mybib}

\begin{thebibliography}{10}

\bibitem{Alradhawi2016New}
{\sc M.~A. Alradhawi and D.~Angeli}, {\em New approach to the stability of
  chemical reaction networks: Piecewise linear in rates lyapunov functions},
  IEEE T. Automat. Contr., 61 (2016), pp.~76--89.

\bibitem{Anderson2015Lyapunov}
{\sc D.~F. Anderson, G.~Craciun, M.~Gopalkrishnan, and C.~Wiuf}, {\em Lyapunov
  functions, stationary distributions, and non-equilibrium potential for
  reaction networks}, Bull. Math. Biol., 77 (2015), pp.~1744--1767.

\bibitem{Angeli2009A}
{\sc D.~Angeli}, {\em A tutorial on chemical reaction network dynamics}, Eur.
  J. Control., 15 (2009), pp.~398--406.

\bibitem{barenholz2017design}
{\sc U.~Barenholz, D.~Davidi, E.~Reznik, Y.~M. Baron, N.~Antonovsky, E.~Noor,
  and R.~Milo}, {\em Design principles of autocatalytic cycles constrain enzyme
  kinetics and force low substrate saturation at flux branch points}, eLife., 6
  (2017).

\bibitem{Cappelletti2018Graphically}
{\sc D.~Cappelletti and B.~Joshi}, {\em Graphically balanced equilibria and
  stationary measures of reaction networks.}, SIAM J. Appl. Dyn. Syst., 17
  (2018), pp.~2246--2175.

\bibitem{Craciun2006Multiple}
{\sc G.~Craciun and M.~Feinberg}, {\em Multiple equilibria in complex chemical
  reaction networks: Ii. the species-reaction graph}, SIAM J. Appl. Math., 66
  (2006), pp.~1321--1338.

\bibitem{decraene2009crosstalk}
{\sc J.~Decraene, G.~G. Mitchell, and B.~Mcmullin}, {\em Crosstalk and the
  cooperation of collectively autocatalytic reaction networks}, IEEE Congr.
  Evol. Comput.(CEC),  (2009), pp.~2249--2256.

\bibitem{Fang2015Lyapunov}
{\sc Z.~Fang and C.~Gao}, {\em Lyapunov function partial differential equations
  for chemical reaction networks: Some special cases}, SIAM J. Appl. Dyn.
  Syst., 18 (2019), pp.~1163--1199.

\bibitem{Feinberg1972Complex}
{\sc M.~Feinberg}, {\em Complex balancing in general kinetic systems}, Arch.
  Ration. Mech. Anal., 49 (1972), pp.~187--194.

\bibitem{Finberg1987Chemical}
{\sc M.~Feinberg}, {\em Chemical reaction network structure and the stability
  of complex isothermal reactors-i. the deficiency zero and deficiency one
  theorems}, Chem. Eng. Sci., 42 (1987), pp.~2229--2268.

\bibitem{Finberg1988Chemical}
{\sc M.~Feinberg}, {\em Chemical reaction network structure and the stability
  of complex isothermal reactors-ii.multiple steady states for networks of
  deficiency one}, Chem. Eng. Sci., 43 (1988), pp.~1--25.

\bibitem{Feinberg1995The}
{\sc M.~Feinberg}, {\em The existence and uniqueness of steady states for a
  class of chemical reaction networks}, Arch. Ration. Mech. Anal., 132 (1995),
  pp.~311--370.

\bibitem{gross2020joining}
{\sc E.~Gross, H.~A. Harrington, N.~Meshkat, and A.~Shiu}, {\em Joining and
  decomposing reaction networks.}, J. Math. Biol.,  (2020), pp.~1--49.

\bibitem{gruning2010regulatory}
{\sc N.~Gruning, H.~Lehrach, and M.~Ralser}, {\em Regulatory crosstalk of the
  metabolic network}, Trends. Biochem. Sci., 35 (2010), pp.~220--227.

\bibitem{hoessly2019stationary}
{\sc L.~Hoessly}, {\em Stationary distributions via decomposition of stochastic
  reaction networks}, arXiv: Probability,  (2019).

\bibitem{hordijk2010autocatalytic}
{\sc W.~Hordijk, J.~Hein, and M.~Steel}, {\em Autocatalytic sets and the origin
  of life}, Entropy, 12 (2010), pp.~1733--1742.

\bibitem{Hordijk2004}
{\sc W.~Hordijk and M.~Steel}, {\em Detecting autocatalytic, self-sustaining
  sets in chemical reaction systems.}, J. Theoret. Biol., 227 (2004),
  pp.~451--461.

\bibitem{HornJackson1972General}
{\sc F.~Horn and R.~Jackson}, {\em General mass action kinetics}, Arch. Ration.
  Mech. Anal., 47 (1972), pp.~81--116.

\bibitem{Johnston2012Dynamical}
{\sc M.~D. Johnston, S.~David, and S.~Gabor}, {\em Dynamical equivalence and
  linear conjugacy of chemical reaction networks: New results and methods},
  MATCH Commun. Math. Comput. Chem., 68 (2012), pp.~443--468.

\bibitem{Johnston2011Linear}
{\sc M.~D. Johnston and D.~Siegel}, {\em Linear conjugacy of chemical reaction
  networks}, J. Math. Chem., 49 (2011), pp.~1263--1282.

\bibitem{Kauffman1995At}
{\sc S.~A. Kauffman}, {\em At home in the universe: The search for laws of
  self-organization and complexity}, Leonardo., 29 (1995).

\bibitem{Ke2019Complex}
{\sc M.~Ke, Z.~Fang, and C.~Gao}, {\em Complex balancing reconstructed to the
  asymptotic stability of mass-action chemical reaction networks with
  conservation laws}, SIAM J. Appl. Math., 79 (2019), pp.~55--74.

\bibitem{khaluf2017the}
{\sc Y.~Khaluf, C.~Pinciroli, G.~Valentini, and H.~Hamann}, {\em The impact of
  agent density on scalability in collective systems: noise-induced versus
  majority-based bistability}, Swarm. Intell., 11 (2017), pp.~155--179.

\bibitem{lopezotin2010the}
{\sc C.~Lopezotin and T.~Hunter}, {\em The regulatory crosstalk between kinases
  and proteases in cancer}, Nat. Rev. Cancer, 10 (2010), pp.~278--292.

\bibitem{Lu2020}
{\sc Y.~Lu and C.~Gao}, {\em Lyapunov function pdes method to the stability of
  some chemical reaction networks.}, arXiv:2005.01146,  (2020).

\bibitem{Plasson2011Autocatalyses}
{\sc R.~Plasson, A.~Brandenburg, L.~Jullien, and H.~Bersini}, {\em
  Autocatalyses}, J. Phys. Chem. A., 115 (2011), p.~8073.

\bibitem{Szederke2011Finding}
{\sc G.~Szederkényi and K.~M. Hangos}, {\em Finding complex balanced and
  detailed balanced realizations of chemical reaction networks}, J. Math.
  Chem., 49 (2011), pp.~1163--1179.

\bibitem{van2013mathematical}
{\sc A.~van~der Schaft, S.~Rao, and B.~Jayawardhana}, {\em On the mathematical
  structure of balanced chemical reaction networks governed by mass action
  kinetics}, SIAM J. Appl. Math., 73 (2013), pp.~953--973.

\bibitem{Virgo2016Complex}
{\sc N.~Virgo, T.~Ikegami, and S.~Mcgregor}, {\em Complex autocatalysis in
  simple chemistries}, Artif. Life., 22 (2016), pp.~1--15.

\bibitem{Wu2020A}
{\sc S.~Wu, Y.~Lu, and C.~Gao}, {\em Lyapunov function partial differential
  equations for stability analysis of a class of chemical reaction networks.},
  in 21st IFAC World Congress in Berlin, to appear, 2020.

\end{thebibliography}

\end{document}